\documentclass[12pt, oneside]{amsart}
\usepackage{amsmath, amssymb, amsfonts, mathtools, tikz-cd, comment, geometry, enumitem, hyperref, fancyhdr, stmaryrd, circuitikz, graphicx, scalerel, mathrsfs, xcolor}
\usetikzlibrary{arrows,calc,matrix}
\usepackage[T1]{fontenc}
\usepackage{lmodern}

\usepackage{setspace} 
\newcommand{\Spec}{\operatorname{Spec}} 
\newcommand{\Pic}{\operatorname{Pic}} 
\newcommand{\Ext}{\operatorname{Ext}} 
\newcommand{\NS}{\operatorname{NS}} 
\newcommand{\Gr}{\operatorname{Gr}} 
\newcommand{\rk}{\operatorname{rk}} 
\newcommand{\iso}{\xrightarrow{\smash{\raisebox{-0.5ex}{$\textstyle\sim$}}}} 
\def\RHom{\mathop{\mathrm{RHom}}\nolimits} 
\newcommand{\bwed}{\mathop{\raisebox{0.25ex}{\scalebox{0.95}{$\bigwedge$}}}\nolimits} 
\DeclareMathOperator{\Db}{D\textsuperscript{\rm b}} 
\DeclareMathOperator{\Br}{Br} 
\newcommand{\vv}{\mathbf{v}} 
\newcommand{\Sym}{\operatorname{Sym}} 

\numberwithin{equation}{section}

\theoremstyle{plain}
\newtheorem{theorem}{Theorem}[section]
\newtheorem{lemma}[theorem]{Lemma}
\newtheorem{proposition}[theorem]{Proposition}
\newtheorem{corollary}[theorem]{Corollary}

\newtheorem{theoremx}{Theorem}

\newtheorem{propositionx}[theoremx]{Proposition}

\theoremstyle{definition}
\newtheorem{definition}[theorem]{Definition}
\newtheorem{example}[theorem]{Example}

\newtheorem{assumption}[theorem]{Assumption}

\theoremstyle{remark}
\newtheorem{remark}[theorem]{Remark}

\definecolor{imperial}{rgb}{0.015625, 0.25, 0.4375}
\definecolor{denim}{rgb}{0.08, 0.38, 0.74}
\hypersetup{
    colorlinks=true,
    linkcolor=imperial,
    citecolor=imperial,
    filecolor=magenta,      
    urlcolor=denim
    }
    
\title{Semi-rigid stable sheaves: a criterion and examples}
\author[Alessio Bottini \& Riccardo Carini]{Alessio Bottini \& Riccardo Carini}

\begin{document}
\onehalfspacing


\begin{abstract}
    Inspired by Mukai's work on K3 surfaces, we introduce and study a notion of \emph{semi-rigidity} for stable sheaves on smooth polarised varieties, designed to capture the existence of stable deformations of direct sums. We show that semi-rigidity is detected by the absence of decomposable elements in the kernel of the Yoneda pairing.
    We apply the resulting criterion to line bundles on smooth projective varieties and to line bundles supported on smooth Lagrangian subvarieties of hyper-K\"ahler manifolds. 
\end{abstract}


\maketitle
\addtocontents{toc}{\protect\setcounter{tocdepth}{-1}}


\thispagestyle{empty}
\section*{Introduction}

Let $\mathscr F$ be a stable sheaf on a smooth polarised variety $(X,h)$, with Mukai vector $\mathbf v\in H^\ast(X,\mathbb Q)$. Denote by $M_{\mathbf v}(X,h)$ the associated good moduli space, whose closed points parametrise polystable sheaves. For every $n\ge 2$, taking direct sums induces a morphism
    \begin{equation}\label{eq:intro:direct-sum}
        \Sym^n M_{\mathbf v}(X,h)\to M_{n\mathbf v}(X,h),\quad \sum_{i=1}^n [\mathscr{F}_i]\mapsto [\mathscr{F}_1\oplus\cdots\oplus \mathscr{F}_n].
    \end{equation}
The guiding problem of this paper is to describe the local geometry of this morphism around the point $[\mathscr F^{\oplus n}] \in  M_{n\mathbf v}(X,h)$, and to determine when stable sheaves arise in its neighbourhood. 

Our approach is deformation-theoretic. We work under the assumption that the differential graded algebra $\RHom^\ast(\mathscr F,\mathscr F)$ is formal and the deformation space $\mathrm{Def}_{\mathscr F}$ is smooth. Under these hypotheses, the deformation theory is governed by the skew-symmetric Yoneda pairing
    \begin{equation*}
        \Upsilon_{\mathscr{F}}:\bwed^2\Ext^1(\mathscr F,\mathscr F)\to \Ext^2(\mathscr F,\mathscr F), \quad e_1\wedge e_2\mapsto e_1\cup e_2.
    \end{equation*}

We begin with the case $n=2$. Under the above assumptions, we say that $\mathscr F$ is \emph{semi-rigid} if every Jordan--H\"older factor of any sufficiently small deformation of $\mathscr{F}\oplus \mathscr{F}$ is a deformation of $\mathscr{F}$. Equivalently, the direct-sum morphism \eqref{eq:intro:direct-sum} is locally surjective at $2[\mathscr{F}]\in \Sym^2 M_{\vv}(X,h)$ onto a neighbourhood of $[\mathscr{F}\oplus \mathscr{F}]\in M_{2\mathbf v}(X,h)$. In particular, $\mathscr{F}\oplus \mathscr{F}$ has no stable deformations.

\begin{theoremx}\label{thm:intro:semirigid}
    A stable sheaf $\mathscr F$ is semi-rigid if and only if $\ker(\Upsilon_{\mathscr{F}})\subseteq \bwed^2 \Ext^1(\mathscr{F}, \mathscr{F})$ contains no non-zero decomposable element.
\end{theoremx}

Our definition of semi-rigid is inspired by Mukai's terminology for stable sheaves on symplectic surfaces \cite{Mukai1987ModuliBundlesK3}. When $(X,h)$ is a symplectic surface, $\vv\in H^\ast(X, \mathbb{Z})$ is primitive, and $h\in \mathrm{Amp}(X)$ is generic, the direct-sum morphism \eqref{eq:intro:direct-sum} is surjective (indeed, an isomorphism) onto $M_{2\mathbf{v}}(X,h)$ if and only if $\vv^2=0$ \cite[Prop.\@\  3.13]{Mukai1987ModuliBundlesK3}. In this case $\Upsilon_{\mathscr{F}}$ is an isomorphism, the moduli space $M_{\vv}(X,h)$ is also a symplectic surface and stable sheaves with Mukai vector $\vv$ \emph{behave like points} on $X$---if $M_{\vv}(X,h)$ admits a universal family, it induces a derived equivalence between $X$ and $M_{\vv}(X,h)$. Skyscraper sheaves are in fact the prototypical example of semi-rigid stable sheaves, as there is an isomorphism of graded algebras $\Ext^\ast(\mathscr{O}_x, \mathscr{O}_x)\cong \bwed^\ast T_x X$ and hence $\ker(\Upsilon_{\mathscr{O}_x})=0$. 

Our next result explains why our notion of semi-rigidity remains useful beyond $n=2$. The key point is that a decomposable element in $\ker(\Upsilon_{\mathscr{F}})$ produces, for every $n\ge 2$, stable points in the affine quadratic local model governing deformations of $\mathscr F^{\oplus n}$, which in turn give rise to stable deformations of $\mathscr F^{\oplus n}$. As a partial converse, we can also show that semi-rigidity allows us to single out an irreducible component through $[\mathscr F^{\oplus n}]$:

\begin{theoremx}\label{thm:intro:split} 
    Let $\mathscr{F}$ be a stable sheaf on $(X,h)$. 
        \begin{enumerate}[label={\upshape(\roman*)}]  
            \item If $\mathscr F$ is semi-rigid, then for every $n\ge 2$ the direct-sum morphism \eqref{eq:intro:direct-sum} is the embedding of an irreducible component in a neighbourhood of $[\mathscr F^{\oplus n}]$.
            \item If $\mathscr{F}$ is not semi-rigid, then for all $n\geq 2$ any neighbourhood of $[\mathscr{F}^{\oplus n}]$ in $M_{n\vv}(X,h)$ contains a stable sheaf.
        \end{enumerate}
\end{theoremx}

Note that even if $\mathscr{F}$ is semi-rigid, we cannot rule out the existence of stable deformations of $\mathscr{F}^{\oplus n}$ lying on irreducible components other than the \emph{split component} singled out above, unless we assume that $M_{n\vv}(X,h)$ is unibranch at $[\mathscr{F}^{\oplus n}]$. This holds, for instance, under the stronger assumption $\ker(\Upsilon_{\mathscr{F}})=0$, see Proposition \ref{prop:ker-Upsilon-zero}.

\subsection*{Semi-rigid line bundles} 
The semi-rigidity condition can be difficult to check in practice.  However, in some cases, the characterisation of Theorem \ref{thm:intro:semirigid} has a geometric interpretation. The first instance where this happens is the case of line bundles on smooth projective varieties: via the classical Castelnuovo--de Franchis theorem---more precisely, Catanese's topological refinement \cite[Thm.\@ 1.10]{Catanese1991}---we can relate the existence of decomposables to the existence of irrational pencils.

\begin{propositionx}\label{intro:prop-CdF}
    Let $X$ be a smooth projective variety. The following are equivalent:
        \begin{enumerate}[label={\upshape(\roman*)}] 
            \item Line bundles on $X$ are semi-rigid.
            \item For any linearly independent $\eta_1, \eta_2\in H^1(X, \mathbb{C})$, we have $\eta_1\cup \eta_2\neq 0\in H^2(X, \mathbb{C})$.
            \item Every morphism $X\to C$ to a curve $C$ of genus $g(C)\geq 2$ is constant.  
        \end{enumerate}
\end{propositionx}

In particular, any \emph{Albanese primitive} projective variety in the sense of \cite{Catanese1991}---those admitting, more generally, no \emph{higher} irrational pencils---satisfies these conditions. This includes all regular varieties, namely those with $q(X)=0$: in this case line bundles are rigid. There are, however, many irregular Albanese primitive varieties: these include, for instance, abelian varieties and a plethora of examples arising from the geography of irregular surfaces of general type \cite{LopesPardini2012GeographySurfaces}.

\subsection*{Semi-rigid Lagrangian line bundles on hyper-K\"ahler manifolds} 
Condition $\mathrm{(ii)}$ in Proposition \ref{intro:prop-CdF} becomes particularly suggestive in the hyper-K\"ahler setting, which motivates much of this paper. Although deformations of a line bundle $\mathscr{L}$ on $X$ are governed by the $(0,\ast)$-Dolbeault cohomology algebra $H^{0,\ast}(X)\coloneqq H^\ast(X,\mathscr{O}_X)$, Proposition \ref{intro:prop-CdF} shows that semi-rigidity is equivalently encoded by the cup product structure on the Betti cohomology algebra $H^\ast(X,\mathbb C)$.

Now let $i\colon Z\hookrightarrow X$ be a smooth Lagrangian embedding into a hyper-K\"ahler manifold $X$.
For any line bundle $\mathscr{L}$ on $Z$, the torsion sheaf $i_\ast  \mathscr{L}$ on $X$ is stable for every polarisation and the deformation space $\mathrm{Def}_{i_\ast  \mathscr{L}}$ is smooth and symplectic \cite[Thm.\@ 8.1]{DonagiMarkman1996Spectral}.
In favourable situations---namely, when $\RHom^\ast(i_\ast  \mathscr{L},i_\ast  \mathscr{L})$ is formal and the local-to-global spectral sequence degenerates---the deformation theory of $i_\ast  \mathscr{L}$ is essentially governed by the Betti cohomology of $Z$.
More precisely, one obtains a multiplicative filtration $F^\ast$ on $\Ext^\ast(i_\ast  \mathscr{L},i_\ast  \mathscr{L})$ whose associated graded satisfies
\begin{equation*}
    \mathrm{gr}_{F} \Ext^\ast(i_\ast  \mathscr{L},i_\ast  \mathscr{L})\cong H^\ast(Z,\mathbb C).
\end{equation*}
In this case, Proposition \ref{intro:prop-CdF} allows us to translate semi-rigidity from $Z$ to $X$:

\begin{propositionx}\label{intro:prop-semirigid-lagrangian}
    Under the above assumptions, $\mathscr{L}$ is semi-rigid on $Z$ if and only if $i_\ast  \mathscr{L} $ is semi-rigid on $X$.
\end{propositionx}

Moduli spaces of sheaves on high-dimensional hyper-K\"ahler manifolds are notoriously hard to understand, especially for the deformations of $i_\ast \mathscr{L}$ whose support is a singular Lagrangian. 
However,  \cite{Bottini2024OG10Stable} provides one example where we can isolate a smooth symplectic connected component.
Let $X\coloneqq F(Y)$ be the Fano variety of lines on a smooth cubic fourfold $Y\subseteq \mathbb{P}^5$.
By \cite{BeauvilleDonagi1985} it is a hyper-K\"ahler fourfold of $\mathrm{K3}^{[2]}$-type. Let $Z\coloneqq F(Y_H)$ be the smooth Lagrangian surface parametrising lines contained in a smooth hyperplane section $Y_H\subseteq Y$. It comes with a Lagrangian embedding $i\colon Z\hookrightarrow X$.  Fix a polarisation $h\in \operatorname{Amp}(X)$ and set $\vv\coloneqq \vv(i_\ast  \mathscr{O}_Z)\in H^\ast(X, \mathbb{Q})$. In  \cite{Bottini2024OG10Stable}, it is shown that, if $Y$ is sufficiently general, the moduli space $M^s_\vv(X,h)$ has a connected component $M^\circ_\vv(X,h)$ through $[i_\ast  \mathscr{O}_Z]$ which is smooth and hyper-K\"ahler of $\mathrm{OG10}$-type.
It is isomorphic to the Laza--Saccà--Voisin compactification \cite{LazaSaccaVoisin2017OG10}  of the intermediate Jacobian fibration of $Y$. 

With such a smooth, stable component at hand, it is natural to look at multiples $M_{n\vv}(X,h)$ for $n\geq 2$, in the spirit of O'Grady's construction \cite{OGrady1999OG10, OGrady2000OG6}, and ask whether they give rise to new irreducible symplectic varieties via suitable (partial) resolutions. In the above setting, however, this strategy fails: using the results above, we show that the sheaves parametrised by $M^\circ_{\vv}(X,h)$ are semi-rigid.

\begin{theoremx}\label{intro:thm-Fano-surface}
    The component $M^\circ_{\vv}(X,h)$ parametrises semi-rigid stable sheaves. Moreover, for all $n\geq 2$ the direct-sum morphism $\Sym^n M^\circ_{\vv}(X,h)\to M_{n\vv}(X,h)$ is an isomorphism in a neighbourhood of the image of the diagonal.
\end{theoremx}
 
This is slightly stronger than Theorem \ref{thm:intro:split} and follows from $\ker(\Upsilon_{i_\ast  \mathscr{O}_Z})=0$, see Proposition \ref{prop:ker-Upsilon-zero}. Consequently, this component is a \emph{primitive symplectic variety} in the sense of \cite{BakkerLehn2022GlobalModuliTheory}, and it admits no symplectic resolution.

\subsection*{Reducible symplectic moduli spaces} As an immediate corollary of Theorem \ref{intro:thm-Fano-surface}, we can show that moduli spaces $M_{\vv}(X,h)$ with fixed Mukai vector $\vv\in H^\ast(X, \mathbb{Q})$ on higher-dimensional hyper-K\"ahler manifolds may not be irreducible, in contrast with the surface case. 

More concretely, pick a smooth cubic fivefold $\widetilde{Y}\subseteq \mathbb{P}^6$ containing $Y$ as a hyperplane section. Let $Z'\coloneqq F_2(\widetilde{Y})$ be the variety of planes contained in $\widetilde{Y}$. By \cite{IlievManivel2008}, intersecting with $Y$ defines a Lagrangian immersion $j\colon Z'\to  X$ with
    \begin{equation*}
        \vv(j_\ast \mathscr{O}_{Z'})=63 \vv(i_\ast  \mathscr{O}_Z)\in H^\ast(X, \mathbb{Q})
    \end{equation*}
and $\dim \Ext^1(j_\ast \mathscr{O}_{Z'}, j_\ast \mathscr{O}_{Z'})=42$, see Section \ref{section:examples}. Thus $[j_\ast \mathscr{O}_{Z'}]\in M^s_{63\vv}(X,h)$ is a smooth point lying on a 42-dimensional component---birational to the irreducible symplectic variety constructed in \cite{liu2025irreducible}---distinct from the 630-dimensional split component given by Theorem \ref{intro:thm-Fano-surface}. In particular, $M_{63\vv}(X,h)$ is not irreducible.

\subsection*{Outline} 
In Section \ref{section:local-structure-pi} we review the local description of the germ of $M_{n\vv}(X,h)$ at $[\mathscr{F}^{\oplus n}]$. This yields an explicit algebraic model for $\mathrm{Def}_{\mathscr{F}^{\oplus n}}$, given by a quadratic cone $\mu^{-1}(0)\subseteq \Ext^1(\mathscr{F}^{\oplus n}, \mathscr{F}^{\oplus n})$ inside the tangent space of the moduli stack $\mathscr{M}_{n\vv}(X,h)$. 

In Section \ref{section:commuting-scheme} we introduce a second quadratic cone $\chi^{-1}(0)\subseteq \Ext^1(\mathscr{F}^{\oplus n}, \mathscr{F}^{\oplus n})$, which can be identified with a commuting scheme for tuples of matrices in $\mathfrak{gl}_n$. In Section \ref{section:local-structure-direct sum} we show there is a $\mathrm{GL}_n$-equivariant inclusion $\chi^{-1}(0)\subseteq \mu^{-1}(0)$. Passing to quotients by $\mathrm{GL}_n$, this inclusion provides a local description of the direct-sum morphism \eqref{eq:intro:direct-sum}. The argument relies crucially on a generalised version of the Chevalley restriction theorem for commuting schemes. 

In Section \ref{section:semirigid}, we introduce the notion of \emph{semi-rigid} stable sheaves, and use the above local description to prove Theorems \ref{thm:intro:semirigid} and \ref{thm:intro:split}. Finally, Section \ref{section:examples} contains examples and applications. We first treat elementary cases, including line bundles and the geometric criterion in Proposition \ref{intro:prop-CdF}, and then turn to the hyper-K\"ahler setting and the main example of Theorem \ref{intro:thm-Fano-surface}. 

\subsection*{Acknowledgments} We wish to thank Borislav Mladenov for helpful correspondence about the algebra structure on \eqref{eq:Ext*-Betticohomology}. This research was funded by the ERC Synergy Grant HyperK, Grant
agreement ID 854361.

\addtocontents{toc}{\protect\setcounter{tocdepth}{1}}
\tableofcontents


\section{Local structure of \texorpdfstring{$\mathscr{M}_\mathbf{v}(X,h) \to M_\mathbf{v}(X,h)$}{Mv(X,h) to Mv(X,h)}}\label{section:local-structure-pi}

Let  $\mathscr{F}$ be a polystable sheaf on a smooth polarised variety $(X,h)$. We view $\mathscr{F}$ as a closed point of the moduli stack $\mathscr{M}_\vv(X,h)$ of Gieseker semistable sheaves on $(X,h)$ with Mukai vector $\vv\coloneqq \vv(\mathscr{F})\in H^\ast(X, \mathbb{Q})$, and we write $G\coloneqq \operatorname{Aut}(\mathscr{F})$ for its stabiliser group. 

By standard deformation theory \cite{ArbarelloSaccaQuiver, KaledinLehnSorger2006Singular, Manetti2022LieMethods}, the marked deformation functor of $\mathscr{F}$ admits a pro-representable hull $\mathrm{Def}_\mathscr{F}\coloneqq \operatorname{Spf}(R)$. Since $G$ is reductive, the natural $G$-action on marked deformations can be lifted to $R$, so that at the level of good moduli spaces we get an isomorphism of completed local rings
    \begin{equation*}
        \hat{\mathscr{O}}_{M_{\vv}(X,h), [\mathscr{F}]}\cong R^G.
    \end{equation*}

Kuranishi theory provides an explicit presentation of $\mathrm{Def}_\mathscr{F}$ as the zero locus of a formal $G$-equivariant map
    \begin{equation}\label{eq:Kuranishi-map}
        \kappa_{\mathscr{F}}\colon \widehat{\Ext}{}^{1}(\mathscr{F},\mathscr{F})\to \Ext^2(\mathscr{F},\mathscr{F}),
    \end{equation}
where $\widehat{\Ext}{}^{1}(\mathscr{F},\mathscr{F})$ denotes the formal completion of $\Ext^1(\mathscr{F}, \mathscr{F})$ at the origin. One may arrange that $\kappa_{\mathscr{F}}$ has no constant or linear term, and the quadratic term is induced by the Yoneda square map
    \begin{equation}\label{eq:Yoneda-square}
        \mu_{\mathscr{F}}\colon \Ext^1(\mathscr{F}, \mathscr{F})\to \Ext^2(\mathscr{F}, \mathscr{F}), \quad e\mapsto e\cup e.
    \end{equation}
In other words, one has 
    \begin{equation*}
        \kappa_{\mathscr{F}}(e)=\mu_{\mathscr{F}}(e)+ \textnormal{higher order terms in }e.
    \end{equation*}

Assume now that the differential graded Lie algebra $\RHom^\ast(\mathscr{F}, \mathscr{F})$ is formal. Then one may choose $\kappa_{\mathscr{F}}$ to equal its quadratic term, i.e.\@ $\kappa_{\mathscr{F}}=\mu_{\mathscr{F}}$. This provides an explicit geometric algebraisation for  $\mathrm{Def}_\mathscr{F}$: it can be identified with the formal completion at $0\in \Ext^1(\mathscr{F}, \mathscr{F})$ of the quadratic cone $\mu_{\mathscr{F}}^{-1}(0)\subseteq \Ext^1(\mathscr{F}, \mathscr{F})$. Accordingly, the quotient stack $[\mu_{\mathscr{F}}^{-1}(0)/G]$ provides an \'etale-local model for the moduli stack $\mathscr{M}_\mathbf{v}(X,h)$, compatible with the good moduli space morphism $\pi\colon \mathscr{M}_\vv(X,h)\to M_{\vv}(X,h)$:

\begin{lemma}\label{lem:local-structure-moduli} 
    There exists an affine $G$-scheme $\Spec A$, containing a $G$-fixed closed point $x\in \Spec A$, and a commutative diagram 
        \begin{equation*}
            \begin{tikzcd}
                ({[}\mu_{\mathscr{F}}^{-1}(0)/G{]},0) \ar[d] & ({[}\Spec A/G{]}, x) \ar[r] \ar[l] \ar[d]   & (\mathscr{M}_{\mathbf{v}}(X,h), \mathscr{F}) \ar[d, "\pi"] \\
                (\mu_\mathscr{F}^{-1}(0) /\!/ G, 0 )& (\Spec A^G, x) \ar[r] \ar[l] & (M_{\mathbf{v}}(X,h), [\mathscr{F}])
            \end{tikzcd}
        \end{equation*}
    in which the horizontal arrows are \'etale and both squares are Cartesian.
\end{lemma}

In other words, the morphism $\pi\colon \mathscr{M}_{\mathbf{v}}(X,h)\to M_{\mathbf{v}}(X,h)$ is, \'etale-locally near $\mathscr{F}$, modelled on the natural map from the quotient stack $[\mu_\mathscr{F}^{-1}(0)/G]$ to the affine GIT quotient $\mu_\mathscr{F}^{-1}(0)/\!/ G$. 

\begin{proof} 
    Using a standard global quotient presentation of the moduli stack $\mathscr{M}_\vv(X,h)$ (via a suitable open subscheme of a Quot scheme), one applies the Luna \'etale slice theorem at the point corresponding to $\mathscr{F}$. This yields a $G$-invariant slice $(S,s)$ and an \'etale morphism $([S/G], s)\to (\mathscr{M}_\mathbf{v}(X,h), \mathscr{F})$. By  \cite[\textsection 4]{ArbarelloSaccaQuiver}, there exists a $G$-equivariant isomorphism of analytic germs between $(S,s)$ and $(\mu_\mathscr{F}^{-1}(0), 0)$. A suitable $G$-equivariant version of Artin approximation then upgrades this analytic identification to an \'etale neighbourhood \cite[\textsection 4.6]{AlperHallRydhLunaEtale}. See also \cite[\textsection 4]{ArbarelloSaccaBridgeland} and \cite[Thm.\@ 5.11]{DavisonPurityCY2} for comparable formulations in the CY2 setting, and \cite[Thm.\@ 1.1]{TodaModuliStacksSemistableSheaves2018} for a more general version without the formality assumption. 
\end{proof}

\begin{remark}[Stable locus]\label{rmk:local-structure-stable-locus} 
    The stable locus $\mathscr{M}^s_\vv(X,h)\subseteq \mathscr{M}_{\vv}(X,h)$ can be characterised as the locus of polystable sheaves with $\mathbb{G}_m$-stabiliser. It follows from the lemma that stable sheaves correspond, in the local model $\mu_{\mathscr{F}}^{-1}(0)$, to points with closed $G$-orbit and stabiliser equal to the diagonal scalar $\mathbb{G}_m$. Equivalently, after rigidifying by $\mathbb{G}_m$, this is the stable locus of $\mu^{-1}(0)$ for the induced action.  
\end{remark}

\begin{remark}[Compatibility with direct sums]\label{rmk:compatibility-direct sum} 
    The above algebraisation of the Kuranishi model is functorial with respect to direct sums. In particular, the natural morphism of deformation functors
        \begin{equation*}
            \mathrm{Def}_{\mathscr{F}}\times \mathrm{Def}_\mathscr{F}\to \mathrm{Def}_{\mathscr{F}\oplus \mathscr{F}}
        \end{equation*}
    induced by taking direct sums, is algebraised, under the formality assumption, by the morphism of quadratic cones
        \begin{equation*}
            \mu_{\mathscr{F}}^{-1}(0)\times \mu_{\mathscr{F}}^{-1}(0)\to \mu_{\mathscr{F}\oplus \mathscr{F}}^{-1}(0)
        \end{equation*}
    induced by the block-diagonal inclusion 
        \begin{equation}\label{block-diagonal-inclusion}
            \Ext^1(\mathscr{F}, \mathscr{F})\oplus \Ext^1(\mathscr{F}, \mathscr{F})\to \Ext^1(\mathscr{F}\oplus \mathscr{F}, \mathscr{F}\oplus \mathscr{F}). 
        \end{equation}
    Passing to good moduli spaces, the induced morphism on completed local rings agrees with that coming from the corresponding map of GIT quotients. By \cite[\href{https://stacks.math.columbia.edu/tag/0GDX}{Tag 0GDX}]{StacksProject}, this formal compatibility can be realised after passing to suitable \'etale neighbourhoods. In particular, \'etale-locally near the points $([\mathscr{F}],[\mathscr{F}])\in M_{\vv}(X, h)\times M_{\vv}(X, h)$ and $[\mathscr{F}\oplus \mathscr{F}]\in M_{2\vv}(X,h)$, the direct-sum morphism
        \begin{equation}
            M_{\vv}(X, h)\times M_{\vv}(X, h)\to M_{2\vv}(X,h), \quad ([\mathscr{F}_1], [\mathscr{F}_2]) \mapsto [\mathscr{F}_1\oplus \mathscr{F}_2]
        \end{equation}
    is modelled on the induced morphism between the quotients 
        \begin{equation*}
            \mu_{\mathscr{F}}^{-1}(0)/\!/\mathrm{Aut}(\mathscr{F})\times \mu_{\mathscr{F}}^{-1}(0)/\!/\mathrm{Aut}(\mathscr{F})\to \mu_{\mathscr{F}\oplus \mathscr{F}}^{-1}(0)/\!/ \mathrm{Aut}(\mathscr{F}\oplus \mathscr{F}). 
        \end{equation*}
\end{remark}

\begin{remark}[Moduli of representations and quiver varieties]\label{rmk:quiver-varieties} 
After choosing a polystable decomposition of $\mathscr{F}$, one can canonically encode the $G$-representation $\Ext^{1}(\mathscr{F},\mathscr{F})$ by an \emph{Ext-quiver} $Q$ with dimension vector $\mathbf{d}$, yielding a natural $G$-equivariant identification
\begin{equation*}
    \Ext^{1}(\mathscr{F},\mathscr{F})\cong \mathrm{Rep}_{Q}(\mathbf{d})
\end{equation*}
with the space of representations of $Q$, or equivalently of the path algebra $\mathbb{C}Q$ (see, for instance, \cite{ArbarelloSaccaQuiver, TodaModuliStacksSemistableSheaves2018}). Under the formality assumption of Lemma \ref{lem:local-structure-moduli}, a representative for the deformation space $\mathrm{Def}_{\mathscr{F}}$ is cut out by the Yoneda square map
\begin{equation*}
    \mu_{\mathscr{F}}\colon \Ext^{1}(\mathscr{F},\mathscr{F})\to \Ext^{2}(\mathscr{F},\mathscr{F}), \quad e\mapsto e\cup e.
\end{equation*}
Equivalently, the equations $\mu_{\mathscr{F}}=0$ may be viewed as a finite collection of quadratic relations among the arrows of $Q$, namely as a two-sided quadratic ideal $I\subset \mathbb{C}Q$, and one may identify
\begin{equation*}
    \mu_{\mathscr{F}}^{-1}(0)\cong \mathrm{Rep}_{(Q,I)}(\mathbf{d}),
\end{equation*}
where the latter parametrises representations of the \emph{bound quiver} $(Q,I)$, or equivalently of the associative quadratic algebra $\mathbb{C}Q/I$.

The main advantage of this explicit algebraisation of $\mathrm{Def}_{\mathscr{F}}$ is that it allows one to exploit the toolkit available for moduli of representations of associative algebras and quiver varieties \cite{ArtinAzumaya1969, King1994ModuliRep}. For instance, when $(X,h)$ is a polarised symplectic surface, Serre duality identifies $Q$ with a doubled quiver and $\mu_{\mathscr{F}}$ with the moment map for the natural symplectic structure on $\Ext^{1}(\mathscr{F},\mathscr{F})$, so that the affine quotient $\mu_{\mathscr{F}}^{-1}(0)\!/\!/G$ is a \emph{Nakajima quiver variety} \cite{ArbarelloSaccaQuiver}. This perspective underlies the local analysis of singularities of moduli spaces of sheaves on symplectic surfaces carried out in \cite{KaledinLehnSorger2006Singular}.
\end{remark}


\section{The commuting scheme}\label{section:commuting-scheme}

Let $n\geq 2$ and let $V$ be a complex vector space of dimension $d\geq 1$. Set $G\coloneqq\operatorname{GL}_n$ with Lie algebra $\mathfrak{g}\coloneqq \operatorname{Lie}(G)$.  The \textit{commuting scheme} is the closed subscheme 
    \begin{equation*}
        \mathfrak{C}(\mathfrak{g}, V) \subseteq V\otimes \mathfrak{g}
    \end{equation*}
defined as the zero fibre of the commutator morphism
    \begin{equation}\label{eq:commutator-morphism}
        \chi \colon V\otimes \mathfrak{g} \to \bwed^2 V \otimes \mathfrak{g}, \quad \sum_{i}e_i\otimes A_i\mapsto \frac{1}{2}\sum_{i,j}(e_i\wedge e_j)\otimes [A_i, A_j].
    \end{equation}
After choosing a basis of $V$ we may identify $V\otimes \mathfrak{g}\cong \mathfrak{g}^d$; under this identification, the $\mathbb{C}$-points of $\mathfrak{C}(\mathfrak{g}, V)$ are given exactly by $d$-tuples $(A_1, \dots, A_d)$ of pairwise commuting matrices. 

When $d=1$ the commuting condition is vacuous and $\mathfrak{C}(\mathfrak{g}, V)=\mathfrak{g}$. For $d\geq 2$, however, the geometry of $\mathfrak{C}(\mathfrak{g}, V)$ is notoriously subtle; even for pairs, reducedness remains open, while for larger $d$ reducibility and non-reduced components do occur \cite{JelisiejewSivic2022ComponentsSingularitiesQuotCommutingMatrices}. 

The adjoint action of $G$ on $\mathfrak{g}$ induces a natural action on $V\otimes \mathfrak{g}$ by simultaneous conjugation, and this preserves $\mathfrak{C}(\mathfrak{g},V)$. While the scheme $\mathfrak{C}(\mathfrak{g}, V)$ itself can be complicated, its affine GIT quotient $\mathfrak{C}(\mathfrak{g}, V)/\!/G$ is considerably better behaved: in this setting one has a multivariate analogue of the Chevalley restriction theorem.

Let $T\subseteq G$ be the standard maximal torus of diagonal matrices, with Lie algebra $\mathfrak{t}\subseteq \mathfrak{g}$. Write $N\coloneqq N_{G}(T)$ for the normaliser and $W\coloneqq N/T$ for the Weyl group, which we identify with $\mathfrak{S}_n$ via permutation matrices.  The linear subspace $V\otimes \mathfrak{t}\subseteq V\otimes \mathfrak{g}$ parametrises $d$-tuples of commuting diagonal matrices, hence it lies in $\mathfrak{C}(\mathfrak{g}, V)$. The resulting $N$-equivariant closed embedding
    \begin{equation}\label{eq:block-diagonal-immersion}
        V\otimes \mathfrak{t}\to \mathfrak{C}(\mathfrak{g}, V)
    \end{equation}
induces a morphism between the affine GIT quotients $V\otimes \mathfrak{t}/\!/ N \to \mathfrak{C}(\mathfrak{g}, V) /\!/ G$. Since $T$ acts trivially on $V\otimes \mathfrak{t}$, the $N$-action factors through $W$, and we may equally view this as a morphism
    \begin{equation}\label{eq:Chevalley-morphism}
        V\otimes \mathfrak{t}/\!/ W \to \mathfrak{C}(\mathfrak{g}, V) /\!/ G.
    \end{equation}
For $d=1$, this is the classical Chevalley isomorphism for $\mathrm{GL}_n$, reflecting the fact that the ring of invariants $\mathbb{C}[\mathfrak{g}]^{G}$ is generated by the elementary symmetric polynomials in the eigenvalues. 

\begin{theorem}[{\cite[Thm.\@ 3]{Vaccarino2007Commuting}}]\label{thm:Vaccarino-Chevalley}
    The natural morphism 
        \begin{equation}\label{eq:Chevalley-isomorphism}
            V\otimes \mathfrak{t} /\!/ W \iso \mathfrak{C}(\mathfrak{g}, V) /\!/ G
        \end{equation}
    is an isomorphism. In particular, the quotient $\mathfrak{C}(\mathfrak{g}, V) /\!/ G$ is reduced and normal.  
\end{theorem}

\begin{remark} 
    One may replace $\operatorname{GL}_n$ by an arbitrary connected complex reductive group $G$. In that generality, the morphism \eqref{eq:Chevalley-morphism} is known to be the normalisation of $(\mathfrak{C}(\mathfrak{g}, V) /\!/ G)_{\text{red}}$ \cite[\textsection 6]{Hunziker1997ClassicalInvariantTheory}, and it is conjectured to be an isomorphism \cite[Conj. 3.1]{ChenNgo2020Hitchin}. For $d=1$, this recovers the classical Chevalley restriction theorem for $G$. 
\end{remark}

Upon fixing a basis of $\mathfrak{t}$, one gets an isomorphism 
    \begin{equation*}
        V\otimes \mathfrak{t}/\!/ W \iso \Sym^n V
    \end{equation*}
with the $n$-fold symmetric product of $V$.  Under this identification, the quotient morphism $\mathfrak{C}(\mathfrak{g}, V)\to \mathfrak{C}(\mathfrak{g}, V)/\!/ G$ sends a commuting $d$-tuple to the unordered set of its joint eigenvalues.
Since a $d$-tuple of commuting matrices can be simultaneously triangularised, we can also think of an element in $\mathfrak{C}(\mathfrak{g}, V)$ as an upper triangular matrix with entries in $V$.
The joint eigenvalues then are simply the diagonal entries, which are well defined up to reordering.

Let $(V\otimes \mathfrak{t})_{\text{reg}}\subseteq V\otimes \mathfrak{t}$ be the open subset consisting of diagonal matrices with distinct joint eigenvalues. Its image $U\coloneqq (V\otimes \mathfrak{t})_{\text{reg}} /\!/ W$ in the quotient gets identified with the complement of the big diagonal $\Delta\subseteq \Sym^n V$. Note that, for $d \geq 2$, the open $U$ is precisely the smooth locus of $V\otimes \mathfrak{t}/\!/ W\cong \mathfrak{C}(\mathfrak{g}, V) /\!/G$.

\begin{proposition}\label{prop:geometric-quotient-away-diagonal} 
    The base-change $[\mathfrak{C}(\mathfrak{g}, V)/G]_U\to U\subseteq\mathfrak{C}(\mathfrak{g}, V) /\!/ G$ is a coarse moduli space. 
\end{proposition}

Recall this means the morphism $[\mathfrak{C}(\mathfrak{g}, V)/G]_U\to U$ induces a bijection on isomorphism classes of $\mathbb{C}$-points. Equivalently, the morphism of schemes $\mathfrak{C}(\mathfrak{g}, V)\to \mathfrak{C}(\mathfrak{g}, V) /\!/ G$ is a geometric quotient over the regular locus $U$ \cite[\textsection 4.2.1]{Alper2025StacksModuli}. Note that the GIT-theoretic $G$-stable locus of $\mathfrak{C}(\mathfrak{g}, V)$ is empty, as the generic stabiliser has positive dimension, even after rigidifying by $\mathbb{G}_m$.  

\begin{proof} 
    A commuting $d$-tuple with distinct joint eigenvalues can be simultaneously diagonalised. The joint eigenvectors are those of a generic linear combination (which has distinct eigenvalues). It follows that the $G$-equivariant open subset $\mathfrak{C}(\mathfrak{g}, V)_{\text{reg}}\subseteq \mathfrak{C}(\mathfrak{g}, V)$ corresponding to $[\mathfrak{C}(\mathfrak{g}, V)/G]_U$ consists exactly of those commuting tuples that are simultaneously diagonalisable with distinct joint eigenvalues. On this open subset, every $G$-orbit is closed and contains a unique $N$-orbit in $(V\otimes \mathfrak{t})_{\text{reg}}$, thus giving a $G$-equivariant isomorphism
        \begin{equation*} 
            G\times^N (V\otimes \mathfrak{t})_{\text{reg}}\iso \mathfrak{C}(\mathfrak{g}, V)_{\text{reg}}.
        \end{equation*}
    Hence we get an isomorphism of quotient stacks over $U$
        \begin{equation*}
            [(V\otimes \mathfrak{t})_{\text{reg}}/N]\iso [\mathfrak{C}(\mathfrak{g}, V)/G]_U.
        \end{equation*}
    Finally $[(V\otimes \mathfrak{t})_{\text{reg}}/N]\to (V\otimes \mathfrak{t})_{\text{reg}}/\!/W\cong U $ is a coarse moduli space since the $T$-action on $(V\otimes \mathfrak{t})_{\text{reg}}$ is trivial. 
\end{proof}


\section{Local structure of \texorpdfstring{$\Sym^n M_\vv(X,h)\to M_{n\vv}(X,h)$}{Symn Mv(X,h) to Mnv}}\label{section:local-structure-direct sum}

For the remainder of the paper we impose the following standing hypothesis.

\begin{assumption}\label{ass:main-assumption} 
    Let $\mathscr{F}$ be a stable sheaf on a smooth polarised variety $(X,h)$ satisfying the following properties:
        \begin{enumerate}[label={\upshape(\roman*)}]
            \item The differential graded algebra $\RHom^\ast(\mathscr{F}, \mathscr{F})$ is associatively formal.
            \item The deformation space $\mathrm{Def}_\mathscr{F}$ is smooth. 
        \end{enumerate}
\end{assumption}

\begin{remark}\label{rmk:on-main-ass}
    In the hyper-K\"ahler setting that motivates much of this paper, condition (i) is known for polystable sheaves on symplectic surfaces and for hyperholomorphic vector bundles on hyper-K\"ahler manifolds \cite{BandieraManettiMeazzini2021Formality, MeazziniOnorati2023Hyperholomorphic}. As explained in Section \ref{section:local-structure-pi}, in more geometric terms, it implies that the deformation space $\operatorname{Def}_\mathscr{F}$ is cut out by quadratic equations. This is not strictly necessary for all our results, but it allows us to work with \'etale neighbourhoods instead of tangent cones. One could also work with the weaker assumption that $\RHom^\ast(\mathscr{F}, \mathscr{F})$ is formal as a differential graded Lie algebra, which still yields quadraticity; however, in that case one would need to assume the relevant formality for $\mathscr{F}^{\oplus n}$ uniformly in $n$ (see the proof of Lemma \ref{lem:local-structure-direct-sum}). Condition (ii) holds for stable sheaves on symplectic surfaces by Mukai \cite{Mukai1984SymplecticStructure}, but is not currently known in general in higher dimension. See Section \ref{section:examples} for some examples. 
\end{remark}

From the discussion in Section \ref{section:local-structure-pi}, under our formality assumption on $\RHom^\ast (\mathscr{F},\mathscr{F})$, the local deformation space $\operatorname{Def}_{\mathscr{F}}$ of a \emph{stable} sheaf $\mathscr{F}$ on $(X,h)$ is smooth if and only if the Yoneda pairing 
    \begin{equation*}
        \Ext^1(\mathscr{F}, \mathscr{F})\otimes \Ext^1(\mathscr{F}, \mathscr{F})\to \Ext^2(\mathscr{F}, \mathscr{F}), \quad e_1\otimes e_2\mapsto e_1 \cup e_2
    \end{equation*}
is alternating, i.e.\@ $\mu_{\mathscr{F}}\equiv 0$. Since $\mathscr{F}$ is stable, $\mathrm{Aut}(\mathscr{F})=\mathbb{G}_m$ acts trivially on $\mu_{\mathscr{F}}^{-1}(0)$, so that $\mu_{\mathscr{F}}^{-1}(0)$ is also a local model for $M_{\vv}(X,h)$ in a neighbourhood of $[\mathscr{F}]$. 

It follows that, under Assumption \ref{ass:main-assumption}, the Yoneda pairing defines a linear map
    \begin{equation}\label{eq:definition-Upsilon}
        \Upsilon_{\mathscr{F}} \colon \bwed^2 \Ext^1(\mathscr{F}, \mathscr{F})\to \Ext^2(\mathscr{F}, \mathscr{F}), \quad e_1\wedge e_2\mapsto e_1\cup e_2. 
    \end{equation} 
Fix $n\geq 2$. Our aim is to understand deformations of $\mathscr{F}^{\oplus n}$ and the geometry of the direct-sum morphism 
    \begin{equation}\label{eq:direct-sum-morphism2}
        \Sym^n M_\vv(X,h)\to M_{n\vv}(X,h), \quad \sum_{i=1}^n [\mathscr{F}_i] \mapsto [\mathscr{F}_1\oplus \cdots \oplus \mathscr{F}_n]
    \end{equation}
in a neighbourhood of the point $n[\mathscr{F}]$. 

As in Section \ref{section:local-structure-pi}, set $G\coloneqq \operatorname{Aut}(\mathscr{F}^{\oplus n})\cong \operatorname{GL}_n$ with Lie algebra $\mathfrak{g}\coloneqq \operatorname{Lie}(G)$.  For each $k\geq 0$ there is a canonical isomorphism of $G$-representations
    \begin{equation*} 
        \Ext^k(\mathscr{F}^{\oplus n}, \mathscr{F}^{\oplus n})\cong  \Ext^k(\mathscr{F}, \mathscr{F})\otimes\mathfrak{g},
    \end{equation*}
where $G$ acts on $\mathfrak{g}$ via the adjoint action. 

Since the Yoneda pairing on $\Ext^1(\mathscr{F}, \mathscr{F})$ is skew-symmetric, we can write $\mu\coloneqq \mu_{\mathscr{F}^{\oplus n}}$ from \eqref{eq:Yoneda-square} as
    \begin{equation}\label{eq:definition-mu}
        \begin{split}
            \mu \colon \Ext^1(\mathscr{F},\mathscr{F})  & \otimes \mathfrak{g} \to \Ext^2(\mathscr{F},\mathscr{F}) \otimes \mathfrak{g} \\
            \sum_i e_i \otimes A_i &\mapsto \frac{1}{2} \sum_{i,j} (e_i \cup e_j) \otimes [A_i,A_j].
        \end{split}
    \end{equation}
Setting $V\coloneqq \Ext^1(\mathscr{F}, \mathscr{F})$ with $\dim V=d$, we observe that  $\mu$ factors through the commutator morphism \eqref{eq:commutator-morphism}, so that we get a commutative diagram
    \begin{equation}\label{eq:crucial-factorisation}
        \begin{tikzcd}[column sep=small]
            V\otimes \mathfrak{g} \ar[rr, "\mu"] \ar[dr, "\chi", swap]  & & \Ext^2(\mathscr{F}, \mathscr{F})\otimes \mathfrak{g} \\
            & \bwed^2 V\otimes \mathfrak{g} \ar[ur, "\Upsilon_{\mathscr{F}} \otimes \mathrm{id}_{\mathfrak{g}}", swap] &
        \end{tikzcd}
    \end{equation}
In particular, using the notation of Section \ref{section:commuting-scheme}, we have a $G$-equivariant closed embedding $\mathfrak{C}(\mathfrak{g}, V)= \chi^{-1}(0)\subseteq \mu^{-1}(0)$. This turns out to be the local model of \eqref{eq:direct-sum-morphism2}:

\begin{lemma}\label{lem:local-structure-direct-sum} 
    There exist pointed affine schemes $(\Spec A,x)$ and $(\Spec B, y)$ and a commutative diagram
        \begin{equation*}
            \begin{tikzcd}
                (\mathfrak{C}(\mathfrak{g}, V)/\!/G, 0) \ar[d]  & (\Spec A, x) \ar[d] \ar[r] \ar[l] &  (\Sym^n M_{\vv}(X,h), n [\mathscr{F}]) \ar[d] \\
                (\mu^{-1}(0)/\!/G, 0)  & (\Spec B, y) \ar[r] \ar[l] &  (M_{n \vv}(X,h), [\mathscr{F}^{\oplus n}])
            \end{tikzcd} 
        \end{equation*}
    where the horizontal arrows are \'etale. In particular \eqref{eq:direct-sum-morphism2} is a closed embedding in a neighbourhood of $[\mathscr{F}^{\oplus n}]\in M_{n\vv}(X,h)$.
\end{lemma}

\begin{proof}
    Since we have an isomorphism $\RHom^\ast(\mathscr{F}^{\oplus n}, \mathscr{F}^{\oplus n})\cong \RHom^\ast (\mathscr{F}, \mathscr{F})\otimes M_{n}(\mathbb{C})$ of differential graded associative algebras, our formality assumption on $\mathscr{F}$ is inherited by $\mathscr{F}^{\oplus n}$,  so we can apply Lemma \ref{lem:local-structure-moduli} to $\mathscr{F}^{\oplus n}$ as well. 

    By Remark \ref{rmk:compatibility-direct sum} the block-diagonal inclusion \eqref{block-diagonal-inclusion}, which gives a local model for the direct-sum morphism
        \begin{equation*}
            \prod_{i=1}^n \mathscr{M}_\vv(X,h)\to \mathscr{M}_{n\vv}(X,h), \quad (\mathscr{F}_1,\dots, \mathscr{F}_n)\mapsto \mathscr{F}_1\oplus \cdots \oplus \mathscr{F}_n
        \end{equation*}
    corresponds precisely to the closed embedding $V\otimes \mathfrak{t}\to\mathfrak{C}(\mathfrak{g}, V)\subseteq V\otimes \mathfrak{g}$, as in \eqref{eq:block-diagonal-immersion}. Passing to affine GIT quotients, we obtain a factorisation
        \begin{equation*}
            V\otimes \mathfrak{t}/\!/ W\to \mathfrak{C}(\mathfrak{g}, V)/\!/ G\to \mu^{-1}(0)/\!/G. 
        \end{equation*}
    The first map is an isomorphism by Theorem \ref{thm:Vaccarino-Chevalley}. Since by Lemma \ref{lem:local-structure-moduli} the symmetric product $\Sym^n M_{\vv}(X,h)$ is \'etale-locally around the point $n [\mathscr{F}]$ modelled on a neighbourhood of the origin in the quotient $V\otimes \mathfrak{t}/\!/ W\cong \mathfrak{C}(\mathfrak{g}, V)/\!/G$, the direct-sum morphism $\Sym^n M_{\vv}(X,h)\to M_{n\vv}(X,h)$ corresponds to the closed embedding of good quotients $\mathfrak{C}(\mathfrak{g}, V)/\!/ G\to \mu^{-1}(0)/\!/G$. This shows that \eqref{eq:direct-sum-morphism2} is unramified at $n[\mathscr{F}]$. Since it is finite, and it is universally injective in a neighbourhood of $n[\mathscr{F}]$ by uniqueness of the Jordan--H\"older factors, we conclude it is a closed embedding in a neighbourhood of $[\mathscr{F}^{\oplus n}]\in M_{n\vv}(X,h)$ \cite[\href{https://stacks.math.columbia.edu/tag/04DG}{Tag 04DG}]{StacksProject}.
\end{proof}


\section{Semi-rigid stable sheaves}\label{section:semirigid}

In this section, we introduce the notion of semi-rigidity for stable sheaves and relate the behaviour of the direct-sum morphisms to the existence of decomposable elements in $\ker(\Upsilon_{\mathscr{F}})$, thus proving Theorems \ref{thm:intro:semirigid} and \ref{thm:intro:split}. 

\begin{definition}
    Let $\mathscr{F}$ be a stable sheaf on a smooth polarised variety $(X,h)$ with Mukai vector $\mathbf{v}\in H^\ast(X, \mathbb{Q})$, satisfying Assumption \ref{ass:main-assumption}. We say $\mathscr{F}$ is \emph{semi-rigid} if every Jordan--H\"older factor of any sufficiently small deformation of $\mathscr{F}\oplus \mathscr{F}$ is a deformation of $\mathscr{F}$. 
\end{definition}

After Lemma \ref{lem:local-structure-direct-sum}, this is equivalent to requiring that the direct-sum morphism
        \begin{equation}\label{eq:direct-sum-n=2}
            \Sym^2 M_{\vv}(X,h)\to M_{2\vv}(X,h)
        \end{equation}
induces a homeomorphism from some neighbourhood of $2[\mathscr{F}]\in \Sym^2 M_{\vv}(X,h)$ onto a neighbourhood of $[\mathscr{F}\oplus \mathscr{F}]\in M_{2\vv}(X,h)$; equivalently, that the image of \eqref{eq:direct-sum-n=2} contains an open neighbourhood of $[\mathscr{F}\oplus \mathscr{F}]$. 

This implies, in particular, that $\mathscr{F}\oplus \mathscr{F}$ has no stable deformations. Since \eqref{eq:direct-sum-n=2} is locally a closed embedding by Lemma \ref{lem:local-structure-direct-sum}, semi-rigidity implies it is a universal homeomorphism around $[\mathscr{F}\oplus \mathscr{F}]$: it is simply the inclusion of the reduced structure $M_{2\vv}(X,h)_{\mathrm{red}}\subseteq M_{2\vv}(X,h)$. 

We are now ready to prove Theorem \ref{thm:intro:semirigid}, relating semi-rigidity to the kernel of the Yoneda pairing \eqref{eq:definition-Upsilon}:

\begin{theorem}[Thm.\@\  \ref{thm:intro:semirigid}]\label{thm:semirigid} 
    Let $\mathscr{F}$ be a stable sheaf on a smooth polarised variety $(X,h)$ satisfying \textnormal{Assumption \ref{ass:main-assumption}}. Then $\mathscr{F}$ is semi-rigid if and only if $\ker(\Upsilon_{\mathscr{F}})$ contains no non-zero decomposable element.
\end{theorem}

\begin{proof}  
    We use the notation established at the beginning of Section \ref{section:local-structure-direct sum}. Assume that \eqref{eq:direct-sum-n=2} is not locally surjective at $2[ \mathscr{F}]$. By Lemma \ref{lem:local-structure-direct-sum}, we deduce that the $\mathrm{GL}_2$-equivariant closed embedding $\mathfrak{C}(\mathfrak{gl}_2, V)\to \mu^{-1}(0)$ is not surjective. 
        
    By fixing a basis of $V=\langle e_1, \dots, e_d \rangle$, this means there exists a $d$-tuple of matrices $\alpha\coloneqq(A_1,\dots, A_d)\in \mathfrak{gl}_2^d$ with $\mu(\alpha)=0$ but $\chi(\alpha)\neq 0$. Now using that $\operatorname{tr}(XY)$ is a non-degenerate bilinear form on $\mathfrak{sl}_2= [\mathfrak{gl}_2,\mathfrak{gl}_2]$, there exists a non-zero element $H\in \mathfrak{sl}_2$ such that contracting $\chi(\alpha)\in \bwed^2 V\otimes \mathfrak{sl}_2$ with $H$ yields a non-zero element
        \begin{equation*}
            \langle \chi(\alpha), H \rangle\coloneqq \sum_{i< j}\operatorname{tr}([A_i,A_j]H) \ e_i\wedge e_j\in \ker (\Upsilon_{\mathscr{F}})\subseteq \bwed^2 V.
        \end{equation*}
    By thinking of it as a skew-symmetric matrix on $V^\ast$, we see that it is decomposable if and only if the rank of the matrix with entries $\operatorname{tr}([A_i, A_j]H)$ is $2$.
    Since for every $i,j$ one has $\operatorname{tr}([A_i, A_j]H)=\operatorname{tr}(A_i[A_j, H])$, this shows that
        \begin{equation*} 
            \rk (\operatorname{tr}([A_i, A_j]H))_{ij} \leq \dim [\alpha, H], 
        \end{equation*}
    where $[\alpha, H]\subseteq \mathfrak{sl}_2$ is the linear subspace generated by the commutators $[A_i, H]$, for $1\leq i\leq d$. One trivially has $[\alpha, H]\subseteq [\mathfrak{gl}_2, H]$, so that $\dim [\alpha, H]\leq \dim \mathcal{O}_H$, where $\mathcal{O}_H$ is the adjoint orbit of $H$. Since any non-zero element $H\in \mathfrak{sl}_2$ has $\dim \mathcal{O}_H=2$, this shows that $\langle \chi(\alpha), H \rangle \in \ker(\Upsilon_{\mathscr{F}})$ is decomposable\footnote{Note that the rough estimate $\dim[\alpha, H]\leq \dim \mathcal{O}_H$ does not help in constructing a low-rank element in $\ker (\Upsilon_{\mathscr{F}})$ when $n\geq 3$, as the minimal adjoint orbit dimension is $2n-2$, so one would need a better estimate for $\dim [\alpha, H]$ for $\alpha\in \mu^{-1}(0)$.}. 
        
    Conversely, assume there exists a non-zero $e_1\wedge e_2\in \ker(\Upsilon_{\mathscr{F}})$ and let $\{X,Y,H\}\subseteq \mathfrak{gl}_2$ be an $\mathfrak{sl}_2$-triple. Then the element
        \begin{equation*}
            \alpha\coloneqq e_1\otimes X+e_2\otimes Y\in V\otimes \mathfrak{gl}_2
        \end{equation*}
    satisfies $\mu(\alpha)=0$ but $\chi(\alpha)=(e_1\wedge e_2)\otimes H\neq 0$. Contracting with a linear form $\phi\in (\bwed^2 V)^\ast$ which is non-zero on $e_1\wedge e_2$, the $\mathrm{GL}_2$-invariant function $\operatorname{tr}\langle \phi, \chi(-) \rangle^k$ does not vanish at $\alpha$ for some $k\geq 2$, but it is identically zero on $\mathfrak{C}(\mathfrak{gl}_2, V)$. This shows that the closed embedding 
        \begin{equation}\label{eq:local-model-direct-sum}
            \mathfrak{C}(\mathfrak{gl}_2, V)/\!/\mathrm{GL}_2\to \mu^{-1}(0)/\!/\mathrm{GL}_2
        \end{equation}
    is not surjective. Since \eqref{eq:local-model-direct-sum} is $\mathbb{G}_m$-equivariant, it is not surjective in an arbitrary neighbourhood of $0\in \mu^{-1}(0)/\!/\mathrm{GL}_2$, and hence we can deduce that $\Sym^2 M_{\vv}(X,h)\to M_{2\vv}(X,h)$ cannot be surjective around $2[\mathscr{F}]\in \Sym^2 M_{\vv}(X,h)$ by Lemma \ref{lem:local-structure-direct-sum}.
\end{proof}

\begin{proposition}\label{prop:generic-vanishing}
    Let $\mathscr{F}$ be a semi-rigid stable sheaf on a smooth polarised variety $(X,h)$, with Mukai vector $\vv\in H^\ast(X, \mathbb{Q})$. Then the following hold:
        \begin{enumerate}[label={\upshape(\roman*)}]
            \item The morphism $\pi\colon \mathscr{M}_{2\vv}(X,h)\to M_{2\vv}(X,h)$ is a coarse moduli space in a neighbourhood of $[\mathscr{F}\oplus \mathscr{F}]$, away from the diagonal. In other words, every semistable sheaf with non-isomorphic Jordan--H\"older factors near $\mathscr{F}$ is polystable. 
            
            \item For $([\mathscr{F}_1], [\mathscr{F}_2])\in M^s_{\mathbf{v}}(X,h)\times M^s_{\mathbf{v}}(X,h)$ in an open neighbourhood of $([\mathscr{F}], [\mathscr{F}])$, with $\mathscr{F}_1\ncong \mathscr{F}_2$ we have
                \begin{equation*}
                    \Ext^1(\mathscr{F}_1, \mathscr{F}_2)=0, \quad \Ext^1(\mathscr{F}_2, \mathscr{F}_1)=0.
                \end{equation*}
        \end{enumerate}
\end{proposition}

\begin{proof}
    Once again, we follow the setup established in Section \ref{section:local-structure-direct sum}. By Theorem \ref{thm:semirigid} and its proof, we know that if $\mathscr{F}$ is semi-rigid, then the $\mathrm{GL}_2$-equivariant closed embedding $\mathfrak{C}(\mathfrak{gl}_2,V)\subseteq \mu^{-1}(0)$ is surjective. Hence we deduce from Proposition \ref{prop:geometric-quotient-away-diagonal} and the commutative diagram
    \begin{equation*}
        \begin{tikzcd}
            {[}\mathfrak{C}(\mathfrak{gl}_2, V)/\mathrm{GL}_2{]} \ar[d]\ar[r] & {[}\mu^{-1}(0)/\mathrm{GL}_2{]} \ar[d] \\
            \mathfrak{C}(\mathfrak{gl}_2, V)/\!/\mathrm{GL}_2 \ar[r] & \mu^{-1}(0)/\!/\mathrm{GL}_2 
        \end{tikzcd}
    \end{equation*}
    that $[\mu^{-1}(0)/\mathrm{GL}_2]\to \mu^{-1}(0)/\mathrm{GL}_2$ is also a coarse moduli space away from the diagonal. This allows us to deduce $\mathrm{(i)}$ from Lemma \ref{lem:local-structure-moduli}. 

    For $\mathrm{(ii)}$, note that if $\mathscr{F}_1$ and $\mathscr{F}_2$ are stable non-isomorphic sheaves with invariants $\mathbf{v}\in H^\ast(X, \mathbb{Q})$ on $(X,h)$, any non-trivial extension 
        \begin{equation*}
            0\to \mathscr{F}_2\to \mathscr{G}\to \mathscr{F}_1\to 0
        \end{equation*}
    gives rise to a strictly semistable sheaf which maps to $[\mathscr{F}_1\oplus \mathscr{F}_2]$ under the morphism $\pi\colon \mathscr{M}_{2\mathbf{v}}(X,h)\to M_{2\mathbf{v}}(X,h)$. If the point $([\mathscr{F}_1],[\mathscr{F}_2])\in M^s_{\mathbf{v}}(X,h)\times M^s_\vv(X,h)$ is sufficiently close to $([\mathscr{F}], [\mathscr{F}])$, this contradicts the fact that $\pi$ is a coarse moduli space near $[\mathscr{F}\oplus \mathscr{F}]$. Hence $\Ext^1(\mathscr{F}_1, \mathscr{F}_2)=0$.
\end{proof}

Using the generic vanishing established in Proposition \ref{prop:generic-vanishing}, we can now address the local structure of $M_{n\vv}(X,h)$ at $[\mathscr{F}^{\oplus n}]$ for $n\geq 3$:

\begin{theorem}[Thm.\@ \ref{thm:intro:split}.(i)]\label{thm:split}
    Let $\mathscr{F}$ be a semi-rigid stable sheaf on a smooth polarised variety $(X,h)$, with Mukai vector $\vv\in H^\ast(X, \mathbb{Q})$. Then for every $n\ge 2$ the direct-sum morphism 
        \begin{equation}\label{eq:direct-sum-thm-split}
            \Sym^n M_{\vv}(X,h)\to M_{n\vv}(X,h)
        \end{equation}
    is the embedding of an irreducible component in a neighbourhood of $[\mathscr F^{\oplus n}] \in M_{n\vv}(X,h)$.
\end{theorem}

\begin{proof} 
    By Lemma \ref{lem:local-structure-direct-sum} we know that \eqref{eq:direct-sum-thm-split} is a closed embedding around $n[\mathscr{F}]$ and $[\mathscr{F}^{\oplus n}]$. We are left to show that it is an isomorphism over an open subset of a neighbourhood of $[\mathscr{F}^{\oplus n}]\in M_{n\vv}(X,h)$. By Proposition \ref{prop:generic-vanishing}, every neighbourhood of $[\mathscr{F}^{\oplus n}]\in M_{n \vv}(X,h)$ contains a point representing a polystable sheaf of the form 
        \begin{equation*} 
            \mathscr{G}\coloneqq \bigoplus_{i=1}^n \mathscr{F}_i,
        \end{equation*}
    with $[\mathscr{F}_i]\in M^s_{\vv}(X,h)$ pairwise non-isomorphic stable sheaves with $\Ext^1(\mathscr{F}_i, \mathscr{F}_j)=0$ for $i\neq j$. We can also assume $[\mathscr{F}_i]\in M^s_{\vv}(X,h)$ are smooth points. We aim to show that \eqref{eq:direct-sum-thm-split} is an isomorphism around $[\mathscr{G}]\in M_{n\vv}(X,h)$. 

    We refer to the local description of $M_{n\vv}(X,h)$ around the point $[\mathscr{G}]$ outlined in Section \ref{section:local-structure-pi}.  As in Remark \ref{rmk:compatibility-direct sum}, the direct-sum morphism
        \begin{equation}\label{eq:direct-sum-def}
            \prod_{i=1}^n \mathrm{Def}_{\mathscr{F}_i}\to \mathrm{Def}_\mathscr{G},
        \end{equation}
   is the formal germ of \eqref{eq:direct-sum-thm-split} after taking invariants by $G\coloneqq \mathrm{Aut}(\mathscr{G})=\mathbb{G}_m^n$. Its
    tangent map is given by the block-diagonal inclusion
        \begin{equation}\label{eq:block-diagonal-inclusion-off-diagonal}
            \bigoplus_{i=1}^n \Ext^1(\mathscr{F}_i, \mathscr{F}_i)\to \Ext^1(\mathscr{G}, \mathscr{G}).
        \end{equation}
    Since by assumption $\Ext^1(\mathscr{F}_i, \mathscr{F}_j)=0$ for $i\neq j$, this is an isomorphism, so that \eqref{eq:direct-sum-def} is also an isomorphism and $\mathrm{Def}_{\mathscr{G}}$ is smooth. Since the $G$-action is trivial, we deduce that \eqref{eq:direct-sum-thm-split} is also an isomorphism in a neighbourhood of $[\mathscr{G}]\in M_{n\vv}(X,h)$. 
\end{proof}

As a concrete application of the philosophy expressed in Remark~\ref{rmk:quiver-varieties}, by means of the standard polystability criterion for representations of finite dimensional algebras, we can prove a partial converse to Theorem \ref{thm:split}, producing stable sheaves arbitrarily close to $[\mathscr{F}^{\oplus n}]\in M_{n\vv}(X,h)$. However, we cannot ensure that these are contained in the component singled out above, unless we are given the extra input that $M_{n\vv}(X,h)$ is unibranch at $[\mathscr{F}^{\oplus n}]$. 

\begin{theorem}[Thm.\@ \ref{thm:intro:split}.(ii)] 
    Let $\mathscr{F}$ be a stable sheaf on a smooth polarised variety $(X,h)$, with Mukai vector $\vv\in H^\ast(X, \mathbb{Q})$, satisfying \textnormal{Assumption \ref{ass:main-assumption}}. If $\mathscr{F}$ is not semi-rigid, then for all $n\geq 2$ any neighbourhood of $[\mathscr{F}^{\oplus n}]$ in $M_{n\vv}(X,h)$ contains a stable sheaf.  
\end{theorem}

\begin{proof}
    The proof is a tightening of the construction in the proof of Theorem \ref{thm:semirigid}. We start with a non-zero $e_1\wedge e_2\in \ker(\Upsilon_{\mathscr{F}})$ and want to produce an element $\alpha\in \mu^{-1}(0)$ which is stable for the rigidified action of $\mathrm{PGL}_n$ on $\mu^{-1}(0)$. Since we can rescale $\alpha$ by keeping it stable, in order to make it arbitrarily close to $0\in \mu^{-1}(0)$, this allows us to produce a stable sheaf in $M_{n\vv}(X,h)$ around $[\mathscr{F}^{\oplus n}]$ by Lemma \ref{lem:local-structure-moduli} and Remark \ref{rmk:local-structure-stable-locus}. 
    
    As in the proof of Theorem \ref{thm:semirigid}, we start with an $\mathfrak{sl}_2$-triple $\{X,Y,H\}\subseteq \mathfrak{gl}_n$ whose associated representation is irreducible. This can always be arranged by picking a regular nilpotent $X\in \mathfrak{gl}_n$ and extending it to an $\mathfrak{sl}_2$-triple by the Jacobson--Morozov theorem; since the Jordan type of $X$ is a single block, the resulting $\mathfrak{sl}_2$-module is irreducible, and any two such triples are $\mathrm{GL}_n$-conjugate \cite[\textsection 3.7]{ChrissGinzburgRepTheory2010}. Setting 
        \begin{equation*}
            \alpha\coloneqq e_1\otimes X+e_2\otimes Y\in V\otimes \mathfrak{gl}_n,
        \end{equation*}
    we have that $\alpha$ has closed $\mathrm{GL}_n$-orbit in $\mu^{-1}(0)$ by Weyl's theorem on complete reducibility and the standard polystability criterion for representations of finite dimensional algebras in terms of semisimplicity \cite{ArtinAzumaya1969}. Finally, Schur's lemma implies that the stabiliser of $\alpha$ is $\mathbb{G}_m$, since the associated representation is irreducible. 
\end{proof}

The following is an elementary numerical criterion to ensure the existence of stable sheaves in $M_{n\vv}(X,h)$ around $[\mathscr{F}^{\oplus n}]$:

\begin{corollary}\label{cor:numerical-criterion-non-split} 
    Under \textnormal{Assumption \ref{ass:main-assumption}}, if 
        \begin{equation}\label{eq:numerical-condition-on-ker-Yoneda}
            \dim \ker (\Upsilon_{\mathscr{F}}) \geq \binom{\dim \Ext^1(\mathscr{F}, \mathscr{F})-2}{2}+1,
        \end{equation}
    then $\ker (\Upsilon_{\mathscr{F}})$ contains a non-zero decomposable element, so $\mathscr{F}$ is not semi-rigid. 
\end{corollary}

\begin{proof} 
    Using the above notation, the numerical assumption \eqref{eq:numerical-condition-on-ker-Yoneda} ensures that $\mathbb{P}\ker(\Upsilon_{\mathscr{F}}) \cap \operatorname{Gr}(2,V)$ is not empty in $\mathbb{P}(\bwed^2 V)$. In other words, we can find $e_1, e_2\in V$ such that $0\neq e_1\wedge e_2\in \ker(\Upsilon_{\mathscr{F}})$.
\end{proof}

We finish this section with a strengthening of Theorem \ref{thm:split}, under the additional assumption that $\Upsilon_{\mathscr{F}}$ is injective:

\begin{proposition}\label{prop:ker-Upsilon-zero}
    Let $\mathscr{F}$ be a stable sheaf on a smooth polarised variety $(X,h)$, with Mukai vector $\mathbf{v}\in H^\ast(X, \mathbb{Q})$, satisfying \textnormal{Assumption \ref{ass:main-assumption}}. If $\ker(\Upsilon_{\mathscr{F}})=0$, then for all $n\geq 2$ the direct-sum morphism
        \begin{equation}\label{eq:direct-sum-upsilon-inj}
            \Sym^n M_{\vv}(X,h)\to M_{n\vv}(X,h)
        \end{equation}
    is an isomorphism in a neighbourhood of $[\mathscr{F}^{\oplus n}]\in M_{n\vv}(X,h)$. In particular, $M_{n\vv}(X,h)$ is reduced and normal at $[\mathscr{F}^{\oplus n}]$. 
\end{proposition}

\begin{proof}
    With the notation of Section \ref{section:local-structure-direct sum}, if $\Upsilon_{\mathscr{F}}$ is injective, we see immediately from diagram \eqref{eq:crucial-factorisation} that the $G$-equivariant closed embedding $\mathfrak{C}(\mathfrak{g}, V)\subseteq \mu^{-1}(0)$ is an equality. Lemma \ref{lem:local-structure-direct-sum} then implies that \eqref{eq:direct-sum-upsilon-inj} is a closed embedding \emph{and} \'etale  at $n[\mathscr{F}]\in \Sym^n M_{\vv}(X,h)$. 
\end{proof}

\begin{example}[Symplectic surfaces]\label{ex:surfaces-case} 
    As pointed out in Remark \ref{rmk:on-main-ass}, if $(X,h)$ is a symplectic surface then any stable sheaf $\mathscr{F}$ satisfies Assumption \ref{ass:main-assumption}. By Serre duality, $\Ext^2(\mathscr{F}, \mathscr{F})$ is one-dimensional, hence the linear map 
        \begin{equation*}
            \Upsilon_{\mathscr{F}}\colon \bwed^2 \Ext^1(\mathscr{F}, \mathscr{F})\to \Ext^2(\mathscr{F}, \mathscr{F})
        \end{equation*}
    is injective if and only if $\dim \Ext^1(\mathscr{F}, \mathscr{F})=2$, i.e.\@ $\vv(\mathscr{F})^2=0$. Assuming $\vv\in H^\ast(X, \mathbb{Z})$ is primitive and $h\in \mathrm{Amp}(X)$ is generic, Proposition \ref{prop:ker-Upsilon-zero} gives a different proof of Mukai's result that if $\vv^2=0$, then $\Sym^n M_{\vv}(X,h)\to M_{n\vv}(X,h)$ is an isomorphism for all $n\geq 2$ \cite[Prop.\@ 3.13]{Mukai1987ModuliBundlesK3}. 

    The cases $\dim \Ext^1(\mathscr{F}, \mathscr{F})>2$ all fall into the range covered by Corollary \ref{cor:numerical-criterion-non-split}, as $\ker(\Upsilon_{\mathscr{F}})$ is a hyperplane in $\bwed^2 \Ext^1(\mathscr{F}, \mathscr{F})$. Of course, in these cases we know the stable locus $M^s_{n\vv}(X,h)$ is nonempty by \cite{Yoshioka2001ModuliAbelian, KaledinLehnSorger2006Singular}, hence \eqref{eq:direct-sum-morphism2} cannot be surjective. In fact, it is an irreducible component of the singular locus of $M_{n\vv}(X,h)$ \cite{KaledinLehnSorger2006Singular}.
\end{example}


\section{Examples and applications}\label{section:examples}
In this section, we review several concrete examples of semi-rigid stable sheaves. After briefly discussing the simplest cases (point-like objects) we turn to the more interesting question of determining when a line bundle is semi-rigid on a smooth projective variety.
We show that this condition is related to the variety being \emph{Albanese primitive}, a topological property studied by Catanese over thirty years ago \cite{Catanese1991}.
Building on this geometric characterisation, we then study the analogous question for line bundles supported on Lagrangian subvarieties of a hyper-K\"ahler manifold.
We focus on the case of the variety of lines on a cubic fourfold, proving Theorem \ref{intro:thm-Fano-surface} from the introduction. 
We note that while rigid sheaves are trivially semi-rigid, we focus on the non-rigid case, where the deformation theory interacts non-trivially with the direct-sum morphisms.

\subsection*{Point-like objects}
The prototypical examples of semi-rigid sheaves are given by skyscraper sheaves on a smooth projective variety $X$. For any $x\in X$, we have that $\mathscr{O}_x$ is stable with respect to any polarisation $h\in \mathrm{Amp}(X)$ and satisfies Assumption \ref{ass:main-assumption}. Moreover, we have an isomorphism of graded algebras
    \begin{equation*}
        \Ext^\ast(\mathscr{O}_x, \mathscr{O}_x)\cong \bwed^\ast T_x X. 
    \end{equation*}
In fact, if $\mathsf{pt} \in H^\ast (X,\mathbb{Q})$ denotes the class of a point, one clearly has an isomorphism $X\cong M_{\mathsf{pt}}(X,h)$ induced by $\mathscr{O}_{\Delta}$ on $X\times X$, and accordingly for any $n\geq 2$ we get isomorphisms
    \begin{equation*}
        \Sym^nX\iso M_{n\mathsf{pt}}(X,h), \quad \sum_{i=1}^n x_i\mapsto [\mathscr{O}_{x_1}\oplus \cdots \oplus \mathscr{O}_{x_n}]. 
    \end{equation*}

After \cite[Prop.\@ 1.4]{BudurZhang2019Formality}, all \emph{stable sheaves} which are images of skyscraper sheaves under derived equivalences also satisfy Assumption \ref{ass:main-assumption} and are semi-rigid. These include for instance:
\begin{enumerate}[label={\upshape(\roman*)}] 
    \item Line bundles on an abelian variety $A$: the Poincaré bundle induces a derived equivalence $\Db(A)\iso \Db(\Pic^0(A))$.
    \item Topologically trivial line bundles on smooth fibres of a Beauville--Mukai system with a section $M$.
    The relative Poincaré sheaf, constructed by Arinkin \cite{Arinkin2013AutodualCompactified}, induces an autoequivalence of $\Db(M)$ sending the skyscraper sheaf of a point on a smooth fibre to such a line bundle. 
    \item Stable sheaves $\mathscr{F}$ on a symplectic surface $(X,h)$ with $\vv(\mathscr{F})$ primitive, $h\in \mathrm{Amp}(X)$ generic, and $\vv(\mathscr{F})^2=0$: any twisted universal family induces a derived equivalence $\Db(X)\cong \Db(M_{\vv}(X,h), \theta)$, for a suitable $\theta\in \Br(M_{\vv}(X,h))$ \cite{Mukai1999DualityK3}. 
    \item Markman's atomic vector bundles on $\mathrm{K3}^{[n]}$-type hyper-K\"ahler manifolds constructed in \cite[Theorem 1.6]{Markman2024StableRank1}.  
\end{enumerate}

\begin{remark}
    In fact, all topologically trivial line bundles on a Lagrangian torus $A \subseteq X$ in a hyper-K\"ahler manifold $X$ are semi-rigid, not just the ones in $\textnormal{(ii)}$. 
    Indeed, formality of $i_\ast \mathscr{L}$ holds by \cite{Mladenov2024Formality}, and there is an isomorphism of graded algebras 
    \[
        \Ext^\ast(i_\ast \mathscr{L},i_\ast \mathscr{L}) \cong H^\ast(A,\mathbb{C}) \cong \bwed^\ast H^1(A,\mathbb{C}),
    \]
    by \cite[Corollary 3.4.4]{Mladenov2024Formality}.
\end{remark}

\subsection*{Semi-rigid line bundles}
In the case of line bundles on a smooth projective variety $X$, the semi-rigidity criterion of Theorem \ref{thm:semirigid} admits an especially concrete geometric interpretation. 

For a line bundle $\mathscr{L}\in \mathrm{Pic}(X)$, we have that $\mathscr{L}$ is stable for any polarisation $h\in \mathrm{Amp}(X)$. Moreover, the differential graded associative algebra $\RHom^\ast(\mathscr{L}, \mathscr{L})\cong \mathrm{R}\Gamma^\ast(\mathscr{O}_X)$ controlling deformations of $\mathscr{L}$ is formal \cite{NeisendorferTaylor1978DolbeaultHomotopyTheory}, and hence isomorphic to the $(0,\ast)$-Dolbeault cohomology algebra $H^{0,\ast}(X)=H^\ast(X, \mathscr{O}_X)$. Moreover, the deformation space $\mathrm{Def}_\mathscr{L}$ is smooth, so line bundles fall within the scope of Assumption \ref{ass:main-assumption}.

The classical Castelnuovo--de Franchis theorem shows that decomposable elements in $\ker(\Upsilon_{\mathscr{L}})$ correspond to irrational pencils on $X$. Moreover, by a result of Catanese \cite{Catanese1991}, this is in fact a topological property of $X$, detected by the cup product structure on the Betti cohomology algebra $H^\ast(X, \mathbb{C})$:

\begin{proposition}[Prop.\@ \ref{intro:prop-CdF}]\label{prop-CdF}
    Let $X$ be a smooth projective variety. The following are equivalent:
        \begin{enumerate}[label={\upshape(\roman*)}] 
            \item Line bundles on $X$ are semi-rigid.
            \item For any linearly independent $\eta_1, \eta_2\in H^1(X, \mathbb{C})$, we have $\eta_1\cup \eta_2\neq 0\in H^2(X, \mathbb{C})$.
            \item Every morphism $X\to C$ to a curve $C$ of genus $g(C)\geq 2$ is constant.  
        \end{enumerate}
\end{proposition}

\begin{proof} Let $\mathscr{L}$ be a line bundle on $X$. By Theorem \ref{thm:semirigid}, $\mathrm{(i)}$ corresponds to no decomposable elements in the kernel of the Yoneda product
    \begin{equation}\label{eq:Upsilon-line-bundles}
        \Upsilon_{\mathscr{L}} \colon \bwed^2 H^1(X, \mathscr{O}_X)\to H^2(X, \mathscr{O}_X). 
    \end{equation}
Via complex conjugation, this turns into the non-existence of decomposable elements in the kernel of the wedge product of differentials
    \begin{equation*}
        \overline{\Upsilon}_{\mathscr{L}}\colon \bwed^2 H^0(X, \Omega^1_X)\to H^0(X, \Omega^2_X),
    \end{equation*}
which is equivalent to $\mathrm{(iii)}$ by Castelnuovo--de Franchis. The equivalence with the topological condition $\mathrm{(ii)}$ is proved in \cite[Thm.\@ 1.10]{Catanese1991}.
\end{proof}

\begin{example}[Albanese primitive varieties]
    Following \cite[Def.\@ 1.20]{Catanese1991}, one says a smooth projective variety $X$ is \emph{Albanese primitive} if it admits no \emph{higher irrational pencils}. These are fibrations $X\to Y$ where $Y$ is a normal variety of Albanese general type, and are detected by higher-dimensional isotropic subspaces for the cup product on $H^1(X, \mathbb{C})$ \cite[Thm.\@ 2.25]{Catanese1991}. In particular, Albanese primitive varieties satisfy the equivalent conditions of Proposition \ref{prop-CdF}. Typical sources of examples include:
        \begin{enumerate}[label={\upshape(\roman*)}] 
            \item Regular varieties with $q(X)=0$. In this case line bundles are in fact \emph{rigid}, in the sense that $\Ext^1(\mathscr{L}, \mathscr{L})=0$. 
            \item Abelian varieties and related constructions: smooth ample complete intersections of dimension $\geq 2$ in abelian varieties, fibrations over an abelian variety with regular generic fibre \cite[Thm.\@ 5.13]{Fujino2005RemarksAlgebraicFiberSpaces} (for instance, high symmetric products of curves or moduli spaces of stable vector bundles on curves).
            \item Many examples among irregular surfaces of general type \cite{LopesPardini2012GeographySurfaces}.
        \end{enumerate}
\end{example}

\subsection*{Semi-rigid Lagrangian line bundles} 

The geometric characterisation above can almost effortlessly be transferred to the setting of Lagrangian line bundles on hyper-K\"ahler manifolds, provided one works under a few favourable hypotheses. 

Let $i\colon Z\hookrightarrow X$ be the embedding of a smooth Lagrangian submanifold into a hyper-K\"ahler manifold $X$, and let $\mathscr{L}\in \Pic(Z)$. The sheaf $i_\ast  \mathscr{L}$ is a pure sheaf supported on $Z$, hence stable with respect to any polarisation. Moreover, its deformation space $\mathrm{Def}_{i_\ast  \mathscr{L}}$ is smooth and symplectic \cite[Thm.\@ 8.1]{DonagiMarkman1996Spectral}. In fact, relative compactified Jacobians of families of Lagrangian subvarieties---often providing birational models for moduli spaces of Lagrangian line bundles---are a promising source of irreducible symplectic varieties \cite{DonagiMarkman1996Spectral, Bottini2024OG10Stable, sacca2024compactifying, liu2025irreducible}.  

To align this situation with Assumption \ref{ass:main-assumption}, we assume that the differential graded algebra $\RHom^\ast(i_\ast  \mathscr{L}, i_\ast  \mathscr{L})$ is formal.
This is not as straightforward as it was for line bundles, but it holds in favourable cases.
For instance, this happens if there exists a derived equivalence $ \Db(X) \iso \Db(X',\theta)$ sending $i_\ast  \mathscr{L}$ to a twisted hyperholomorphic bundle \cite{MeazziniOnorati2025}; further examples are discussed in \cite{Mladenov2024Formality}.  

Moreover, since we want to still be able to use the geometric criterion in Proposition \ref{prop-CdF} in this context, we further assume that the local-to-global spectral sequence
    \begin{equation}\label{eq:local-to-global-spectral-sequence}
        E^{p,q}_2\coloneqq H^q(Z, \Omega^p_Z) \Rightarrow \Ext^{p+q}(i_\ast  \mathscr{L}, i_\ast  \mathscr{L})
    \end{equation}
degenerates at $E_2$. This allows us to compare the graded algebra $\Ext^\ast(i_\ast  \mathscr{L}, i_\ast \mathscr{L})$ with the Dolbeault cohomology of $Z$. 

\begin{example} \label{ex:Atiyah-class-vanishing}
    By \cite{ArinkinCaldararu2012SelfIntersectionFibration, mladenov2025degeneration}, a sufficient condition for \eqref{eq:local-to-global-spectral-sequence} to degenerate is that $\mathscr{L}$ extends to the first infinitesimal neighbourhood of $Z$ in $X$. In particular, this holds whenever the Atiyah class $\mathrm{at}_\mathscr{L}\in \Ext^1(\mathscr{L}, \mathscr{L}\otimes \Omega^1_Z)$ vanishes; for line bundles, this is equivalent to $c_1(\mathscr{L})=0\in H^1(Z, \Omega^1_Z)$, hence for $\mathscr{L}\in \Pic^0(Z)$. We refer to \cite{mladenov2025degeneration} for a more precise criterion and further sufficient conditions. 
\end{example}

Since the spectral sequence \eqref{eq:local-to-global-spectral-sequence} is compatible with the Yoneda product \cite[\textsection 1.2]{mladenov2025degeneration}, each vector space $\Ext^k(i_\ast  \mathscr{L}, i_\ast  \mathscr{L})$ carries a finite decreasing filtration $F^\ast \Ext^k(i_\ast  \mathscr{L}, i_\ast  \mathscr{L})$ whose associated graded pieces are
    \begin{equation*}
        \mathrm{gr}_F^p \Ext^{k}(i_\ast  \mathscr{L}, i_\ast  \mathscr{L})\cong \ H^{k-p}(Z,\Omega_Z^p),\quad 0\leq p\leq k.
    \end{equation*}
Multiplicativity makes these filtrations compatible with the algebra structure, so the associated graded identifies with the Dolbeault cohomology algebra, or equivalently with the Hodge-bigraded Betti cohomology algebra of $Z$:
    \begin{equation}\label{eq:Ext*-Betticohomology}
        \mathrm{gr}_F\Ext^\ast (i_\ast  \mathscr{L}, i_\ast  \mathscr{L})\cong \bigoplus_{p,q} H^q(Z, \Omega^p_Z) \cong H^\ast(Z, \mathbb{C}).
    \end{equation}

In particular, the Yoneda product in degrees $(1,2)$
    \begin{equation*}
        \Upsilon_{i_\ast  \mathscr{L}}\colon \bwed^2 \Ext^1(i_\ast  \mathscr{L}, i_\ast  \mathscr{L}) \to \Ext^2(i_\ast  \mathscr{L}, i_\ast  \mathscr{L})
    \end{equation*}
is a morphism of filtered vector spaces for the induced filtration on $\bwed^2 \Ext^1(i_\ast  \mathscr{L}, i_\ast  \mathscr{L})$. Under \eqref{eq:Ext*-Betticohomology}, its associated graded 
    \begin{equation*}
        \mathrm{gr}_F(\Upsilon_{i_\ast  \mathscr{L}})\colon \bwed^2 H^1(Z, \mathbb{C})\to H^2(Z, \mathbb{C})
    \end{equation*}
is the usual cup product on Betti cohomology, with graded pieces given by the Hodge components. For instance, the top graded piece
    \begin{equation*}
        \mathrm{gr}_F^2 (\Upsilon_{i_\ast  \mathscr{L}})\colon \bwed^2 H^1(Z, \mathscr{O}_Z)\to H^2(Z, \mathscr{O}_Z) 
    \end{equation*}
is exactly the cup product on $H^\ast(Z, \mathscr{O}_Z)$, and hence coincides with the map $\Upsilon_\mathscr{L}$ from \eqref{eq:Upsilon-line-bundles}. The next lemma records the basic comparison between $\Upsilon_{i_\ast  \mathscr{L}}$ and its associated graded.

\begin{lemma}\label{lem:Y-vs-gr(Y)}
    Under the above assumptions, the following hold:
        \begin{enumerate}[label={\upshape(\roman*)}]
            \item If the map $\mathrm{gr}_F(\Upsilon_{i_\ast  \mathscr{L}})$ is injective (respectively surjective), then $\Upsilon_{i_\ast  \mathscr{L}}$ is injective (respectively surjective). 
            \item There exists a non-zero decomposable element in $\ker(\Upsilon_{i_\ast \mathscr{L}})$ if and only if there exists a non-zero decomposable element in $\ker(\mathrm{gr}_F(\Upsilon_{i_\ast  \mathscr{L}}))$.
        \end{enumerate}
\end{lemma}

\begin{proof}
    The first statement is standard filtered linear algebra: injectivity and surjectivity lift from the associated graded. We therefore focus on $\mathrm{(ii)}$. Assume there is a non-zero decomposable element in $\ker(\mathrm{gr}_F(\Upsilon_{i_\ast  \mathscr{L}}))$. By Proposition \ref{prop-CdF} and its proof, one may choose such an element in the $(0,2)$-part, namely in $\ker(\mathrm{gr}_F^2(\Upsilon_{i_\ast  \mathscr{L}}))$, and the latter identifies with $\ker(\Upsilon_{\mathscr{L}})$. Since $F^2\bwed^2 \Ext^1(i_\ast  \mathscr{L}, i_\ast  \mathscr{L})=\bwed^2 F^1 \Ext^1(i_\ast  \mathscr{L}, i_\ast  \mathscr{L})$ is the last step of the induced filtration, this decomposable element lifts canonically to a non-zero decomposable element in $\ker(\Upsilon_{i_\ast  \mathscr{L}})$. Conversely, any non-zero decomposable element of $\ker(\Upsilon_{i_\ast \mathscr{L}})$ has a well-defined leading term with respect to $F^\ast$; after a harmless change of basis in its spanning 2-plane, this leading term is a non-zero decomposable element lying in $\ker(\mathrm{gr}_F(\Upsilon_{i_\ast  \mathscr {L}}))$.
\end{proof}

Combining Theorem \ref{thm:semirigid} and Lemma \ref{lem:Y-vs-gr(Y)}, we can now easily transfer semi-rigidity from line bundles on $Z$ to the corresponding sheaves on $X$:

\begin{proposition}[Prop.\@ \ref{intro:prop-semirigid-lagrangian}]\label{prop:semirigid-lagrangian}
    Under the above assumptions, $\mathscr{L}$ is semi-rigid on $Z$ if and only if $i_\ast  \mathscr{L}$ is semi-rigid on $X$.
\end{proposition}

\begin{example}
    Let $i\colon C\hookrightarrow X$ be the embedding of a smooth curve of genus $g\geq 1$ into a symplectic surface. Proposition \ref{prop:semirigid-lagrangian} shows that $i_\ast  \mathscr{O}_C$ is semi-rigid if and only if $g=1$, which recovers the isotropy condition $0=\vv(i_\ast  \mathscr{O}_C)^2=2g-2$. 
\end{example}

\begin{example}[Albanese primitive Lagrangians]\label{ex:Albanese-primitive-lagrangians}
    Let $i\colon  Z\hookrightarrow X$ be a smooth Lagrangian embedding into a hyper-K\"ahler manifold, and assume that $Z$ is Albanese primitive. Then Propositions \ref{prop-CdF} and \ref{prop:semirigid-lagrangian} imply that, whenever the torsion sheaf $i_\ast  \mathscr{L}$, for $\mathscr{L}\in \Pic^0(Z)$, satisfies Assumption \ref{ass:main-assumption}.(i), it is semi-rigid on $X$. Examples of Albanese primitive Lagrangians include:
        \begin{enumerate}[label={\upshape(\roman*)}]
            \item If $f\colon X\to B$ is a Lagrangian fibration and $Z$ is a smooth fibre, then $Z$ is an abelian variety, hence Albanese primitive.
            \item More generally, smooth components of generic singular fibres of a Lagrangian fibration are Albanese primitive; see \cite[Thm.\@ 1.3]{HwangOguiso2009CharFoliation} for a precise statement. 
        
            A concrete instance comes from Beauville--Mukai systems. Let $(X,h)$ be a polarised K3 surface of genus $g\geq 2$, fix integers $r\geq 2$, $\chi\in \mathbb{Z}$ with $(r, \chi)=1$, and set $\vv\coloneqq (0,rh,\chi)\in H^\ast(X, \mathbb{Z})$. Then the moduli space $M_{\vv}(X,h)$ is a smooth hyper-K\"ahler manifold equipped with the Beauville--Mukai Lagrangian fibration $f\colon M_{\vv}(X,h)\to |rh|$. For a smooth curve $C\in |h|$, the fibre over the non-reduced curve $rC\in |rh|$ is reducible \cite{DonagiEinLazarsfeld1997NilpotentCones}; one component is $Z\coloneqq U_C(r,d)$, the moduli space of stable vector bundles of rank $r$ and degree $d\coloneqq \chi-r(1-g)$ on $C$, viewed as torsion sheaves on $X$. In particular, $Z$ gives a smooth, irregular, Albanese primitive Lagrangian.
            \item Let $X$ be a K3 surface and $C\subseteq X$ be a smooth curve. For $r\geq 2$ the natural embedding $Z\coloneqq\operatorname{Sym}^r C\subseteq \operatorname{Hilb}^r X$ is Lagrangian. Moreover, the cup product $\bwed^k H^1(Z, \mathbb{C})\to H^k(Z, \mathbb{C})$ is injective for $2\leq k \leq r$ \cite[\textsection 6.3]{Macdonald1962SymC}, so $Z$ is Albanese primitive. 
            \item Let $X\coloneqq F(Y)$ be the Fano variety of lines on a smooth cubic fourfold $Y\subseteq \mathbb{P}^5$. The Lagrangian surface $Z\coloneqq F(Y_H)\subseteq X$ parametrising lines contained in a hyperplane section $Y_H\subseteq Y$ is Albanese primitive; indeed, the cup product $\bwed^2 H^1(Z, \mathbb{C})\iso H^2(Z, \mathbb{C})$ is an isomorphism \cite{collino82}.   
        \end{enumerate}
\end{example}

\subsection*{Varieties of lines}
Now we delve deeper into the last example.
Let $Y \subset \mathbb{P}^5$ be a smooth cubic fourfold and write $X\coloneqq F(Y)$ for its Fano variety of lines. It is a hyper-K\"ahler fourfold of $\mathrm{K3}^{[2]}$-type by \cite{BeauvilleDonagi1985}, and it comes with a geometric supply of (possibly singular) Lagrangian subvarieties.

Fix a hyperplane $H \subset \mathbb{P}^5$, and consider the hyperplane section $Y_H\coloneqq Y\cap H$. Let $Z\coloneqq F(Y_H)$ be the surface parametrising lines contained in $Y_H$. 
Then the inclusion $i\colon Z \hookrightarrow X$ is a Lagrangian embedding.

Set $\vv \coloneqq \vv(i_\ast  \mathscr{O}_Z)\in H^\ast(X, \mathbb{Q})$. In \cite{Bottini2024OG10Stable}, it is shown that, for a general cubic fourfold $Y$, there is a distinguished connected component $M^{\circ}_{\vv}(X,h)\subseteq M^s_{\vv}(X,h)$ through $[i_\ast  \mathscr{O}_Z]$, which is hyper-K\"ahler of OG10-type. 

Deformations of the Lagrangian surface $Z=F(Y_H)\subseteq X$ are obtained by varying the hyperplane $H\in (\mathbb{P}^5)^\vee$. Accordingly, the connected component $M^\circ_{\vv}(X,h)$ carries a natural Lagrangian fibration 
    \begin{equation*}
        M^\circ_\vv(X,h)\to (\mathbb{P}^5)^\vee, \quad \mathscr{F}\mapsto \mathrm{supp}(\mathscr{F}).
    \end{equation*}
One can think of it as a compactification of the relative Jacobian of the above family of Lagrangians. In particular, a general point of $M^\circ_{\vv}(X,h)$ is an example of a Lagrangian line bundle on $X$ as discussed above. 

A key input for us is that the sheaves parametrised by $M^\circ_\vv(X,h)$ can be arranged to satisfy Assumption \ref{ass:main-assumption}. Indeed, by \cite[Cor.\@ 6.8]{Bottini2024OG10Stable} one can choose $Y\subseteq \mathbb{P}^5$ so that there exists a derived equivalence 
    \begin{equation}\label{eq:TwistedEquivalence}
        \Db(X)\iso  \Db(X', \theta), \quad \theta\in \Br(X'),  
    \end{equation}
sending every sheaf in $M^\circ_{\vv}(X,h)$ to a $\theta$-twisted hyperholomorphic vector bundle on another hyper-K\"ahler manifold $X'$. In particular, by \cite{MeazziniOnorati2025}, we may choose $Y\subseteq \mathbb{P}^5$ so that the component $M^\circ_{\vv}(X,h)$ parametrises stable sheaves satisfying Assumption \ref{ass:main-assumption}.

\begin{theorem}[Thm.\@ \ref{intro:thm-Fano-surface}]\label{thm:Fano-surface}
    For such a component $M^\circ_{\vv}(X,h)\subseteq M^s_{\vv}(X,h)$, the following hold:
        \begin{enumerate}[label={\upshape(\roman*)}]
            \item $M^\circ_{\vv}(X,h)$ parametrises semi-rigid stable sheaves. 
            \item For all $n\geq 2$ the direct-sum morphism $\Sym^n M^\circ_{\vv}(X,h)\to M_{n\vv}(X,h)$ is the embedding of an irreducible component of dimension $10 n$. 
            \item For all $n\geq 2$, $M_{n\vv}(X,h)$ is reduced and normal along the small diagonal $M^\circ_{\vv}(X,h)\subseteq\Sym^n M^\circ_{\vv}(X,h)$.
        \end{enumerate}
\end{theorem}

\begin{proof}
    We apply Proposition \ref{prop:ker-Upsilon-zero}, so we need that every stable sheaf $\mathscr{F}$ in $M^\circ_{\vv}(X,h)$ is semi-rigid and, in fact, satisfies $\ker(\Upsilon_{\mathscr{F}})=0$. 
    
    First, line bundles $\mathscr{L}$ on smooth deformations of $Z$ satisfy $\ker(\Upsilon_{\mathscr{L}})=0$ by Example \ref{ex:Albanese-primitive-lagrangians}.(iv). For topologically trivial line bundles, the local-to-global spectral sequence \eqref{eq:local-to-global-spectral-sequence} degenerates by Example \ref{ex:Atiyah-class-vanishing}, and Lemma \ref{lem:Y-vs-gr(Y)} then upgrades the injectivity of $\Upsilon_{\mathscr{L}}$ on the associated graded to the injectivity of $\Upsilon_{i_\ast  \mathscr{L}}$. More generally, injectivity of $\Upsilon_{\mathscr{F}}$ is proved in  \cite[Thm.\@ 1.3]{Bottini2024OG10Stable}.
\end{proof}

\begin{remark}
    Owing to our reliance on the formality assumption, Theorem \ref{thm:Fano-surface} is currently restricted to those cubic fourfolds $Y \subset \mathbb{P}^5$ that admit the derived equivalence \eqref{eq:TwistedEquivalence}. While these form a dense subset in the moduli space of smooth cubic fourfolds, we expect the formality condition, and consequently the conclusions of the theorem, to hold in general.
\end{remark}

Assume now that the cubic fourfold $Y$ is itself a hyperplane section of a general cubic fivefold $Y \subseteq \widetilde{Y} \subseteq \mathbb{P}^6$.
The variety of planes $Z'\coloneqq F_2(\widetilde{Y})$ is a smooth surface, and if the setting is general enough, it comes with a natural map 
\begin{equation}\label{eq:i_Z}
    j \colon Z' \to X, \quad \Pi \mapsto \Pi \cap  Y. 
\end{equation}
In \cite{IlievManivel2008}, it is shown that $j$ is a Lagrangian immersion, birational onto its image, and the class of the image is computed to be
\begin{equation}\label{eq:63Times}
[j(Z')] = 63[Z]\in H^{4}(X, \mathbb{Z}).
\end{equation}
In \cite{liu2025irreducible}, the authors construct a singular irreducible symplectic variety by compactifying the relative Jacobian of the family of surfaces $Z'$ (relying on the techniques of \cite{sacca2024compactifying}). 
A natural question is whether this compactification admits a modular interpretation analogous to the one given in \cite{Bottini2024OG10Stable} for the LSV variety.

To address this, the first step is to compute the Mukai vector of the pushforward $j_\ast\mathscr{O}_{Z'}$. 
The most direct approach utilises the theory of \emph{atomic} (or \emph{cohomologically $1$-obstructed}) objects, introduced by Beckmann and Markman \cite{Beckmann2025Atomic,Markman2024StableRank1}. Rather than delving into the full definition of atomic objects, it suffices for our purposes to know that both $i_\ast \mathscr{O}_Z$ and $j_\ast \mathscr{O}_{Z'}$ are atomic (cf.\ \cite{Beckmann2025Atomic,Markman2024StableRank1,guo2025atomic}) and to rely on the following characterisation.

Recall that a \emph{Lagrangian immersion} is a finite unramified morphism $i \colon Z \to X$ with Lagrangian image. 

 \begin{theorem}[{{\cite[Theorem 1.8]{Beckmann2025Atomic}, \cite[Theorem 4.10 and Proposition 4.11]{guo2025atomic}}}]\label{thm:AtomicAndMukaiVector}
     Let $i \colon Z \to X$ be a Lagrangian immersion in a hyper-K\"ahler manifold $X$. 
     Then $i_\ast  \mathscr{O}_Z$ is atomic if and only if 
    \begin{enumerate}[label={\upshape(\roman*)}] 
        \item The pullback $i^\ast \colon H^2(X,\mathbb{Q}) \to H^2(Z,\mathbb{Q})$ has rank $1$.
        \item The first Chern class $c_1(Z)$ lies in the image $\operatorname{Im}i^\ast$. 
    \end{enumerate}
    Moreover, if $\lambda \in \NS(X)_{\mathbb{Q}}$ is a generator of $(\ker i^\ast)^{\perp}$ such that $c_1(Z) = i^\ast\lambda$, we have 
    \[
        \vv(i_\ast  \mathscr{O}_Z) = \exp(\textstyle \frac{\lambda}{2})i_\ast [Z]\in H^\ast(X, \mathbb{Q}).
    \]
 \end{theorem}

\begin{lemma}[{{\cite{Beckmann2025Atomic,Markman2024StableRank1}}}]\label{lem:atomicity3fold}
    The sheaf $i_\ast \mathscr{O}_Z$ is atomic, and we have 
    $$
     \vv(i_\ast  \mathscr{O}_Z) = \exp(-h/2) [Z]\in H^\ast(X, \mathbb{Q}).$$
\end{lemma}

\begin{proof}
    We want to verify the conditions in Theorem \ref{thm:AtomicAndMukaiVector}. Condition (i) is easily checked as follows. 
    Because it is a cohomological statement, we may assume $Y$ to be very general, so that $\mathrm{NS}(X) = \mathbb{Z}.h$.
    Since $Z$ is an immersed Lagrangian subvariety, the restriction map 
    \[ 
    i^\ast \colon H^2(X,\mathbb{Q}) \to H^2(Z,\mathbb{Q})
    \]
    vanishes on the transcendental part $T(X)_{\mathbb{Q}} \subset H^2(X,\mathbb{Q})$, so only the Néron–Severi contributes to the image.
    Since $\NS(X)_{\mathbb{Q}}$ is $1$-dimensional and generated by an ample class, the image $\operatorname{Im}i^\ast$ is $1$-dimensional.
    
    To compute the Mukai vector, notice that $Z$ is cut out by a regular section of $\mathscr{S}^{\vee}$, where $\mathscr{S}$ denotes the restriction to $X$ of the rank $2$ tautological bundle on $\mathrm{Gr}(2,6)$. The adjunction formula gives 
\begin{equation}\label{eq:CanonicalOfHyperplane}
    c_1(Z) = - i^\ast h \in H^2(Z,\mathbb{Z}). 
\end{equation}
where $h \coloneqq \det(\mathscr{S}^{\vee}) \in \mathrm{NS}(X)$ is the Pl\"ucker polarisation. 
The Mukai vector is then given by Theorem \ref{thm:AtomicAndMukaiVector}.
\end{proof}

The following is a slight refinement of  \cite[Proposition 7.5]{guo2025atomic}.

\begin{lemma}\label{lem:AtomicityOfFivefold}
    The sheaf $j_\ast\mathscr{O}_{Z'}$ is atomic, and we have 
    \[
    \vv(j_\ast\mathscr{O}_{Z'}) = \exp(-3h/2) [j(Z')]\in H^\ast(X, \mathbb{Q}). 
    \]
\end{lemma}

\begin{proof}
    We argue as in the above lemma. 
    Condition (i) is checked in the same way by assuming $Y$ to be very general. For (ii), we compute the first Chern class using the adjunction formula. 
    Let $h'$ be the Pl\"ucker polarisation on $\mathrm{Gr}(3,7)$.
    The canonical bundle of $\mathrm{Gr}(3,7)$ is $\omega_{\mathrm{Gr}(3,7)} \cong \mathscr{O}_{\mathrm{Gr}(3,7)}(-7h')$.
    The surface $Z' \subseteq \mathrm{Gr}(3,7)$ is a section of the rank $10$ bundle $\Sym^3\mathscr{S}^{\vee}$, where $\mathscr{S}$ is the rank $3$ tautological bundle on $\mathrm{Gr}(3,7)$. 
    Its determinant is $\mathscr{O}_{\mathrm{Gr}(3,7)}(10h')$, therefore
    \begin{equation}\label{eq:CanonicalOfSurface}
        \omega_{Z'} \cong \mathscr{O}_{Z'}(3h')\in \Pic(Z').
    \end{equation}

    If $Y$ is general enough, the surface $Z' \subset \mathrm{Gr}(3,7)$ does not intersect $\Gr(3,6)$, as otherwise $Y$ would contain a plane. 
    Therefore, we have a commutative diagram
        \[\begin{tikzcd}
	        {Z'} & {X} \\
	        {\mathrm{Gr}(3,7)\setminus\mathrm{Gr}(3,6)} & {\mathrm{Gr}(2,6)}
	        \arrow["{j}", from=1-1, to=1-2]
	        \arrow[hook, from=1-1, to=2-1]
	        \arrow[hook, from=1-2, to=2-2]
	        \arrow["\pi", from=2-1, to=2-2]
        \end{tikzcd}\]
    where $\pi$ is given by intersection with a $6$-dimensional space. 
    The map $\pi$ is a linear projection under the Pl\"ucker embedding, therefore $\pi^\ast(h) = h'$.   
    Combining this with \eqref{eq:CanonicalOfSurface}, it follows that $c_1(Z') = j^\ast(-3h)$.
    Using Theorem \ref{thm:AtomicAndMukaiVector}, we deduce $\vv(j_\ast \mathscr{O}_{Z'}) = \exp(-3h/2)[j(Z')]$.
\end{proof}

\begin{corollary}\label{cor:ReducibilityFano}
   If the cubic $Y\subseteq \mathbb{P}^5$ is chosen as in Theorem \ref{thm:Fano-surface}, then the moduli space $M_{63\vv}(X,h)$ is reducible.  
\end{corollary}

\begin{proof}
By Theorem \ref{thm:Fano-surface}, the direct-sum morphism $\Sym^{63}M^\circ_{\vv}(X,h)\to M_{63\vv}(X,h)$ is the embedding of an irreducible component of dimension $630$.
On the other hand, the sheaf $j_\ast \mathscr{O}_{Z'}(h) \coloneqq j_\ast \mathscr{O}_{Z'} \otimes \mathscr{O}_X(h)$ is torsion-free of rank $1$ on its integral support, therefore it is stable with respect to any polarisation. 
Since $[j(Z')] = 63[Z] \in H^4(X,\mathbb{Z})$, Lemmas \ref{lem:atomicity3fold} and \ref{lem:AtomicityOfFivefold} imply that
    \begin{equation}
        \vv(j_\ast \mathscr{O}_{Z'}(h))=63 \vv(i_\ast  \mathscr{O}_Z)\in H^\ast(X, \mathbb{Q}).
    \end{equation}
Moreover, $j_\ast\mathscr{O}_{Z'}(h)$ has a $42$-dimensional deformation space by  \cite[Corollary 7.6]{guo2025atomic}. Hence it lies on another irreducible component of $M_{63\vv}(X,h)$ different from the split component. 
\end{proof}

\begin{remark}
   The 42-dimensional component $M^{\dagger}_{63\vv}(X,h)$ singled out above is birational to the irreducible symplectic variety constructed in \cite{liu2025irreducible}. However, it remains an open question whether these two varieties are in fact isomorphic. A related and intriguing problem is to describe the boundary points of this component, specifically, whether they parametrise stable or strictly polystable sheaves.
   A more approachable question is whether the component $M^{\dagger}_{63\vv}(X,h)$ and $\Sym^{63} M^\circ_{\vv}(X,h)$ intersect. 
   If they do intersect, it follows from Theorem \ref{thm:Fano-surface} that the intersection is disjoint from the small diagonal of $\Sym^{63} M^\circ_{\vv}(X,h)$.
\end{remark}

\pagestyle{plain}
\bibliographystyle{alpha}
\bibliography{references}
\end{document}